%% file: main.tex
\PassOptionsToPackage{dvipsnames}{xcolor}
\documentclass[onefignum,onetabnum]{siamart171218}
\usepackage{tikz}
\usetikzlibrary{patterns}
\usepackage[dvipsnames]{xcolor}
\usepackage[textwidth=15cm]{geometry}
\usepackage{subcaption}
\usepackage{caption}

\input{ex_shared}
\renewcommand{\epsilon}{\varepsilon}

\ifpdf
\hypersetup{
  pdftitle={Adaptive randomized pivoting for column subset selection},
  pdfauthor={}
}
\fi

\newcommand{\rev}[1]{#1}

\myexternaldocument{ex_supplement}

\begin{document}

\maketitle

\begin{abstract}
We derive a new adaptive leverage score sampling strategy for solving the 
Column Subset Selection Problem (CSSP). The resulting algorithm, called Adaptive Randomized Pivoting, can be viewed as a randomization of Osinsky's recently proposed deterministic algorithm for CSSP.
It guarantees, in expectation, an approximation error that matches the optimal existence result in the Frobenius norm. Although the same guarantee can be achieved with volume sampling, our sampling strategy is much simpler and less expensive.
To show the versatility of Adaptive Randomized Pivoting, we apply it to select indices in the Discrete Empirical Interpolation Method, in cross/skeleton approximation of general matrices, and in the Nystr\"om approximation of symmetric positive semi-definite matrices. In all these cases, the resulting randomized algorithms are new and they enjoy bounds on the expected error that match -- or improve -- the best known deterministic results. A derandomization of the algorithm for the Nystr\"om approximation results in a new deterministic algorithm with a rather favorable error bound.
\end{abstract}



\input{introduction}

\input{algorithm}

\input{analysis}

\input{deim}

\input{crossapproximation}

\input{cholesky}

\input{experiments}

\section{Conclusions}

As demonstrated by theoretical results and preliminary numerical experiments, our newly proposed Adaptive Randomized Pivoting is a simple yet effective strategy for addressing a range of linear algebra problems that involve index selection. Our strategy requires a priori (row) subspace information, which can be obtained by, e.g., sketching the data once. For column subset selection and cross approximation, having access to such subspace information is arguably unavoidable for any reasonably accurate method, in the absence of further prior knowledge on the data. For these two problems, Adaptive Randomized Pivoting is an attractive alternative to existing deterministic methods, which require additional access to the data beyond a simple sketch, and existing inexpensive randomized methods, which require oversampling to attain equally good theoretical bounds. DEIM is arguably the strongest case for our new strategy; in this case, the required subspace information comes naturally with the application, and we attain better bounds than all existing approaches while remaining oblivious to the data to be approximated. {The close connection to DPP sampling deserves further exploration and can lead to further improvements, such as the acceleration of Algorithm~\ref{alg:ARP} by rejection sampling~\cite{Barthelme2023}.} For the case of Nystr\"om approximation, the previously proposed randomly pivoted Cholesky does not require a priori subspace information, thanks to the positive definiteness assumption on the data. This renders Adaptive Randomized Pivoting unattractive for this problem,  despite offering simpler and possibly tighter error bounds. On the other hand, derandomization leads to a new deterministic method that is significantly better than existing deterministic approaches in terms of error bounds.

\paragraph{{Acknowledgements}} {We thank Anil Damle, Ethan N. Epperly, Joel A. Tropp, and Robert J. Webber for sharing valuable insights on this work. We also thank the referees for providing
helpful feedback.}

\bibliographystyle{siamplain}
\bibliography{references}
\end{document}

%% file: ex_shared.tex
\usepackage{lipsum}
\usepackage{amsfonts}
\usepackage{graphicx}
\usepackage{epstopdf}
\usepackage{algorithmic}
\ifpdf
  \DeclareGraphicsExtensions{.eps,.pdf,.png,.jpg}
\else
  \DeclareGraphicsExtensions{.eps}
\fi

\newsiamremark{hypothesis}{Hypothesis}
\crefname{hypothesis}{Hypothesis}{Hypotheses}
\newsiamthm{claim}{Claim}

\headers{Adaptive randomized pivoting for column subset selection}{A. Cortinovis and D. Kressner}

\title{Adaptive randomized pivoting for column subset selection, \\ DEIM, and low-rank approximation
}

\author{Alice Cortinovis\thanks{Department of Computer Science, University of Pisa, Pisa, Italy\rev{, and member of the INdAM/GNCS research group}  (\email{alice.cortinovis@unipi.it})}
\and Daniel Kressner\thanks{Institute of Mathematics, EPFL Lausanne, Lausanne, Switzerland (\email{daniel.kressner@epfl.ch})}.}

\usepackage{amsopn}

\makeatletter
\newcommand*{\addFileDependency}[1]{
  \typeout{(#1)}
  \@addtofilelist{#1}
  \IfFileExists{#1}{}{\typeout{No file #1.}}
}
\makeatother

\newcommand*{\myexternaldocument}[1]{%
    \externaldocument{#1}%
    \addFileDependency{#1.tex}%
    \addFileDependency{#1.aux}%
}

\newcommand{\R}{\mathbb{R}}
\newcommand{\Vopt}{V_{\mathsf{opt}}}

\newcommand{\Id}{\mathsf{Id}}

\newtheorem{example}[theorem]{Example}

%% file: introduction.tex
\section{Introduction}

The Column Subset Selection Problem (CSSP) is a classical linear algebra problem that connects to a variety of fields, including theoretical computer science and statistical learning. Given an $m\times n$ matrix $A$, CSSP aims at selecting $r \ll \min\{m,n\}$ column indices $J = (j_1,\ldots, j_r)$ such that the {span of the} corresponding columns $A(:,J)$ approximately {contains \emph{every} column} of $A$. Letting $\Pi_J$ denote the orthogonal projector onto the column space of $A(:,J)$, a common way to measure the approximation error is to consider
\begin{equation} \label{eq:measure}
 \| A - \Pi_J A\|_F,
\end{equation}
where $\|\cdot \|_F$ denotes the Frobenius norm of a matrix. CSSP is closely connected to a multitude of other matrix approximations and factorizations, including CUR decompositions~\cite{Mahoney2009}, interpolative decompositions~\cite{Voronin2017}, (pseudo-)skeleton approximation~\cite{Goreinov1997}, adaptive cross approximation~\cite{Bebendorf2000}, Gausian elimination with pivoting~\cite{Townsend2013}, pivoted/rank-revealing QR decompositions~\cite{Gu1996}, pivoted Cholesky decompositions~\cite{Harbrecht2012}, Nyström approximation~\cite{Williams2000}, and the discrete empirical interpolation method (DEIM)~\cite{Barraul2004,Chaturantabut2010}.
In this work, we present a simple randomized method for solving CSSP. Our method satisfies an error bound that is optimal in expectation. Its principle design carries over to several of the matrix approximations and factorizations mentioned above, resulting in new methods with favorable error bounds. As we will repeatedly point out, our work has been inspired by Osinsky's remarkable work~\cite{Osinsky2023} on CSSP and related problems. Indeed, we show that the methods from~\cite{Osinsky2023} can be recovered from our methods through derandomization. 

Because the column space of $\Pi_J A$ has dimension at most $r$, the approximation error~\eqref{eq:measure} cannot be smaller than the best rank-$r$ approximation error:
\begin{equation} \label{eq:bestapproximation}
 \|A -  A \Vopt \Vopt^T\|^2_F = \big(\sigma_{r+1}^2(A) + \ldots + \sigma_n^2(A) \big) \le \| A - \Pi_J A\|^2_F,
\end{equation}
where $\sigma_1(A)\ge \sigma_2(A) \ge \ldots \ge \sigma_n(A) \ge 0$ denote the singular values of $A \in \R^{m\times n}$ and we assume, to simplify the presentation, that $m \ge n$. The matrix $\Vopt \in \R^{n \times r}$ represents {any} orthonormal basis {containing right singular vectors belonging to $r$ largest singular values of} $A$, and $A \Vopt \Vopt^T$ is one way of expressing {a} best rank-$r$ approximation of $A$. Of course, this approximation does not solve CSSP because $A \Vopt$ contains linear combinations of \emph{all} columns of $A$. Somewhat surprisingly, it turns out that there is always a column index set $J$ with an approximation error not much larger than the lower bound~\eqref{eq:bestapproximation}. Specifically, Deshpande et al.~\cite{Deshpande2006} showed via volume sampling that there exists $J$ such that
\begin{equation} \label{eq:quasioptCSSP}
  \| A - \Pi_J A\|_F^2 \le (r+1) \big( \sigma_{r+1}^2(A) + \ldots + \sigma_n^2(A) \big),
\end{equation}
and {an example constructed in~\cite[Prop. 3.3]{Deshpande2006} shows} that the factor $r+1$ is tight.
Through derandomization, Deshpande and Rade\-macher~\cite{Deshpande2010} developed a polynomial time, deterministic algorithm for computing a selection $J$ satisfying~\eqref{eq:quasioptCSSP}. This result is quite impressive, in view of the following two facts: (1) A greedy selection of columns, which corresponds to QR with column pivoting~\cite[Sec. 5.4.2]{GolubVanLoan2013} and orthogonal matching pursuit, results in a factor that grows exponentially with $r$~\cite{Gu1996}. (2) 
Finding the optimal $J$ is an NP complete problem~\cite{Shitov2021}. However, the algorithm by Deshpande and Rademacher comes with the disadvantage that its numerical implementation is challenging~\cite{Cortinovis2020} due the need for computing coefficients of (high-degree) characteristic polynomials.

Osinsky~\cite{Osinsky2023} developed a much simpler algorithm for determining $J$ satisfying~\eqref{eq:quasioptCSSP}. In a nutshell, the algorithm proceeds as follows. 
Given an arbitrary orthonormal basis $V \in \R^{n\times r}$, it
progressively turns the orthogonal projection
$
 A V V^T
$
of the rows of $A$ onto $\mathrm{span}(V)$
into an oblique projection $A \widetilde \Pi_J = A(:,J) V(J,:)^{-T} V^T$, which involves a selection of $r$ columns $A(:,J)$ instead of $A V$.
To control the increase of the projection error from $\|A - A V V^T \|_F^2$ to $\|A - A \widetilde \Pi_J\|_F^2$, 
each step of Osinsky's algorithm selects the column index that leads to the smallest increase. As a result,
the total error satisfies
\begin{equation} \label{eq:Osinskyresult}
  \| A - \Pi_J A\|_F^2 \le \|A - A \widetilde \Pi_J\|_F^2 \le (r+1) \|A -  A V V^T\|^2_F.
\end{equation}
Because of~\eqref{eq:bestapproximation}, this matches the tight bound~\eqref{eq:quasioptCSSP} for $V = \Vopt$. Allowing for general $V$ instead of $\Vopt$ offers the flexibility of using a cheap method for obtaining a row space approximation by, e.g., sketching $A^T$ or sampling. In these cases, the need for multiplying $A$ with $V$ in order to compute and control projection errors is a potential  disadvantage of Osinsky's method.

Adaptive Randomized Pivoting, the method proposed in this work, determines the index set $J$ by adaptive leverage score sampling. Each step of the method chooses a column index randomly with probability proportional to the squared row norms of $V$. Before proceeding to the next step, the row of $V$ corresponding to the selected index is removed by an orthogonal projection. This changes the row norms of $V$ and makes the sampling adaptive. After $r$ steps, the selection of $J$ is completed and the columns of $A$ in $J$ are evaluated, but otherwise the matrix $A$ is not directly involved in the process. In particular, the computation of $AV$ is avoided. Our main theoretical result, Theorem~\ref{thm:mainthm}, shows that the inequality~\eqref{eq:Osinskyresult} is satisfied in expectation:
\begin{equation} \label{eq:mainresultintro}
 \mathbb E \| A - \Pi_J A\|_F^2 \le \mathbb E \|A - A \widetilde \Pi_J\|_F^2 = (r+1) \|A -  A V V^T\|^2_F.
\end{equation}
For $V = \Vopt$, this matches a result~\cite[Theorem 1.3]{Deshpande2006} on volume sampling, except that our adaptive sampling strategy is much simpler and cheaper than volume sampling.

Ideas related to Adaptive Randomized Pivoting have been presented several times in the literature; see~\cite{DongMartinsson2023} and the references therein. For example, the subspace sampling method by Drineas et al.~\cite{Drineas2006} samples the columns of $A$ according to the (approximate) row norms of $\Vopt$ without updating the sampling probabilities. Sampling $\mathcal O(\epsilon^{-2} r \log r)$ columns suffices to attain the error bound $(1+\epsilon)^2 \cdot \|A -  A \Vopt \Vopt^T\|^2_F$ with high probability (whp). An improvement of subspace sampling has been presented by Boutsidis et al.~\cite{Boutsidis2009}, which first samples $\mathcal O(r \log r )$ columns according to a (non-adaptive) probability distribution that involves the row norms of $\Vopt$ and the row norms of the projection error. It then deterministically subselects exactly $r$ columns. The output of the algorithm satisfies the error bound $\mathcal O(r^2 \log r) \cdot  \|A -  A \Vopt \Vopt^T\|^2_F$ whp. 
Adaptive sampling strategies for CSSP have been introduced in~\cite{Chen2023,Deshpande2006}, which sample in each step one or more columns from a  probability distribution that is determined by the column norms of $A$ with the previously selected columns removed by orthogonal projection. The resulting approximations satisfy error bounds that rely on oversampling, especially if the best rank-$r$ approximation error is small. Adaptive Randomized Pivoting combines adaptive sampling with subspace sampling. Although we are not aware of it, it is difficult to exclude that such a combination has never been mentioned before in the large literature on CSSP and related problems. At least, such an approach is not mentioned in the recent literature~\cite{Belhadji2020,Chen2023} on randomized methods for CSSP and the result~\eqref{eq:mainresultintro} appears to be unknown.

The rest of this paper is organized as follows. In Section~\ref{sec:cssp}, we describe Adaptive Randomized Pivoting for CSSP and prove the error bound~\eqref{eq:mainresultintro}. In the subsequent sections, we demonstrate how this index sampling strategy can be applied to other matrix approximations and factorizations. Because Adaptive Randomized Pivoting is oblivious to $A$, it is a natural choice for DEIM, which will be discussed in Section~\ref{sec:DEIM}. The resulting randomized algorithm for DEIM enjoys an error bound (in expectation) that is significantly better than error bounds for existing DEIM index selection strategies.
In Section~\ref{sec:crossapproximation}, we apply Adaptive Randomized Pivoting to compute a cross approximation (also called skeleton approximation)
\begin{equation} \label{eq:crossapproximation}
 A \approx A(:,J)A(I,J)^{-1}A(I,:),
\end{equation}
where each of the index sets $I,J$ has cardinality $r$.
In expectation, the squared Frobenius norm error returned by our method is bounded by $(r+1)^2\cdot \|A -  A V V^T\|^2_F$. In Section~\ref{sec:spsd}, we consider the case when $A$ is symmetric positive semi-definite (SPSD), for which~\eqref{eq:crossapproximation} with $I = J$ is sometimes called 
pivoted Cholesky decomposition or Nyström approximation. A trick coined Gram correspondence~\cite{EpperlyBlog} relates the SPSD case to 
CSSP and allows us to conveniently analyze the error when $I$ is obtained by applying Adaptive Randomized Pivoting to $V$. When $V = \Vopt$, Corollary~\ref{cor:SPSD} implies that the trace norm error is bounded in expectation by the best rank-$r$ approximation error multiplied by $r+1$. This is often more favorable than the error bound presented in~\cite{Chen2023} for RPCholesky, an adaptive sampling strategy based on the diagonal of the progressively updated matrix $A$.
A suitable derandomization of the randomized methods presented in Section~\ref{sec:cssp},~\ref{sec:DEIM}, and~\ref{sec:crossapproximation} recovers methods presented by Osinsky~\cite{Osinsky2023}. These methods satisfy the error bounds deterministically but they also require additional access to the data when building the index sets. This limitation is particularly relevant for DEIM, where the data to be approximated is usually not known in advance.
In the SPSD case, the derandomized method presented in Section~\ref{sec:SPSDderandomized} appears to be new. The purpose of the \emph{preliminary} numerical experiments included in Section~\ref{sec:experiments} is mainly to validate our methods and illustrate their potential. The fine-tuning and efficient implementation of our algorithms merit a separate discussion, beyond the scope of this work.

%% file: algorithm.tex
\section{The Column Subset Selection Problem} \label{sec:cssp}
This section considers CSSP, that is, finding a column index set $J$ of cardinality $r$ such that~\eqref{eq:measure} is small. Our proposed algorithm is presented in Section~\ref{sec:ARP}. In Section~\ref{sec:analysis} we show that the algorithm gives, in expectation, a quasi-optimal set of indices, and in Section~\ref{sec:derandomizedCSSP} we discuss how the deterministic algorithm from~\cite{Osinsky2023} can be derived from a derandomization of our proposed algorithm.

\subsection{Adaptive Randomized Pivoting}\label{sec:ARP}

Our algorithm for CSSP assumes the availability of 
an orthonormal basis $V \in \R^{n \times r}$ that represents a good approximation to the row space of $A$, in the sense that 
\begin{equation} \label{eq:rowspaceapprox}
  \|A - A VV^T\|_F \approx 0.
\end{equation}
When $V = \Vopt$ (that is, $V$ contains {right singular vectors belonging to $r$ largest singular values} of $A$), the minimal error~\eqref{eq:bestapproximation} is attained.

To determine the first index $j_1$, our algorithm performs leverage score sampling, that is, $j_1$ is chosen randomly according to the row norms of $V$:
\begin{equation} \label{eq:leveragesampling}
 \mathbb P\{ j_1 = j \} = p_j := \|V(j,:)\|_2^2 / \|V\|_F^2, \quad j = 1,\ldots, n.
\end{equation}
Before proceeding to the next index, the selected row $V(j_1,:)$ is removed by orthogonal projection:
\begin{equation} \label{eq:orthprojection}
 V_{{1}} \gets V (\Id-\Pi_{V(j_1,:)}) = V - V V(j_1,:)^\dagger V(j_1,:),
\end{equation}
where $\Id$ denotes the identity matrix (of suitable size) and $\dagger$ denotes the pseudoinverse of a matrix.
The second index $j_2$ is randomly chosen according to~\eqref{eq:leveragesampling}, with the sampling probabilities adapted by using the row norms of the updated $V_{{1}}$ from~\eqref{eq:orthprojection}. Then the selected row $V_{{1}}(j_2,:)$ is removed by orthogonal projection, {obtaining $V_2 \gets V_1 (\Id-\Pi_{V_1(j_2,:)}) = V_1 - V_1 V_1(j_2,:)^\dagger V_1(j_2,:)$,} and so on. After $r$ steps of the described procedure, all indices $j_1,\ldots,j_r$ have been determined (and $V_{{r}}$ has been reduced to zero).

\begin{algorithm}
\textbf{Input: } Matrix $V \in \R^{n \times r}$ with orthonormal columns defining a row space approximation~\eqref{eq:rowspaceapprox} 

\textbf{Output: }Indices $J=(j_1, \ldots, j_r)$  defining a column subset selection

    \begin{algorithmic}[1]
    \STATE{{Initialize $V_0 = V$ and $J_0 = ()$}}
    \FOR{$k = 1,\ldots,r$}
        \STATE{Set $p_j = \|V_{k-1}(j,{:})\|_2^2/(r-k+1)$ for $j = 1,\ldots,n$}\label{line:probnew}
        \STATE{Sample index $j_k$ according to probabilities $p_j$}
        \STATE{{Set $J_k \leftarrow (J_{k-1}, j_k)$}} 
        \STATE{Update $V_k \leftarrow {V_{k-1} \left ( \Id - V_{k-1}(j_k,:)^\dagger V_{k-1} (j_k,:) \right )}$ 
        }\label{line:projection}
    \ENDFOR
    \STATE{{Set $J \leftarrow J_r$}}
    \end{algorithmic}
    \caption{{Adaptive Randomized Pivoting for CSSP (ARP -- prototype algorithm)}}
    \label{alg:ARPnew}
\end{algorithm}

{Let us define the matrices
\[
 {E_{J_k} = \begin{bmatrix} e_{j_1} & e_{j_2} & \cdots & e_{j_k} \end{bmatrix} \in \mathbb \R^{n\times k},}
\]
where $e_j$ denotes the $j$th unit vector, for $k = 1,\ldots,r$. 
The following lemma tells us how to write $V_k$ in a more compact form that will be useful for the analysis.
\begin{lemma}\label{lem:propV}
    For the matrices generated by Algorithm~\ref{alg:ARPnew} we have:
    \begin{equation}\label{eq:Vk}
        V_k = V \left (\Id - V^TE_{J_k} (V^T E_{J_k})^\dagger\right ) \quad \text{ for }\quad k = 1, \ldots, r.
    \end{equation}
\end{lemma}
\begin{proof}
By construction, $V_{\ell-1}(j_\ell,:)$ is orthogonal to all the matrices $V_k$ for $k \ge \ell$, hence $V_k$ is obtained from $V$ by subtracting $k$ rank-$1$ orthogonal projections:
\begin{equation*}
    V_k = V \left (\Id - V_0(j_1,:)^\dagger V_0(j_1,:) - \ldots - V_{k-1}(j_k,:)^\dagger V_{k-1}(j_k,:) \right ).
\end{equation*}
The matrix $V_0(j_1,:)^\dagger V_0(j_1,:) + \ldots + V_{k-1}(j_k,:)^\dagger V_{k-1}(j_k,:)$ is an orthogonal projection of rank at most $k$, which we denote $\widehat \Pi_k \in \mathbb{R}^{r \times r}$ for the sake of this proof, and such that $E_{J_k}^T V \widehat \Pi_k = E_{J_k}^T V$. Therefore, we have that $\widehat \Pi_k = (E_{J_k}^T V)^\dagger E_{J_k}^T V$ is the orthogonal projection onto the rows $E_{J_k}^T V$. This proves~\eqref{eq:Vk}. 
\end{proof}
A consequence of Lemma~\ref{lem:propV} is that $\|V_k\|_F^2 = r-k$ for $k=0,1,\ldots,r$, which implies that $p_1,\ldots,p_n$ in line~\ref{line:probnew} of Algorithm~\ref{alg:ARPnew} define a probability distribution over $\{1,\ldots,n\}$.}

\paragraph{Choice of low-rank approximation}  
The indices $j_1,\ldots, j_k$ returned by {Algorithm~\ref{alg:ARPnew}} {can be used to construct the orthogonal projection $\Pi_J$ onto the corresponding columns of $A$, but we can also consider the following rank-$r$ approximation, which is naturally induced by Algorithm~\ref{alg:ARPnew}:}
\begin{equation}\label{eq:lowrankapprox}
A\approx    A(:,J) V(J,:)^{-T} V^T.
\end{equation}
{This corresponds to an oblique projection onto the chosen columns of $A$; it is more efficient to compute than the orthogonal projection and still enjoys excellent guarantees; see Theorem~\ref{thm:mainthm} below.}

\begin{remark}
    {Algorithm~\ref{alg:ARPnew} is equivalent to taking a sample of $r$ points from a determinantal point process ($r$-DPP) with kernel matrix $VV^T$; see, e.g.,~\cite[Algorithm 2.1]{Barthelme2023}. In this work we provide an alternative approach to the analysis of Algorithm~\ref{alg:ARPnew} and its use for low-rank approximation purposes.} 
\end{remark}

\subsection{{Stable and efficient implementation of ARP via Householder reflections}} 
{We now describe a way to implement Algorithm~\ref{alg:ARPnew} using Householder reflections.} Householder reflectors~\cite[Sec 5.1.2]{GolubVanLoan2013} represent a numerically safe way of implementing the orthogonal projection~\eqref{eq:orthprojection}. Let $Q_1$ be a Householder reflector that transforms $V(j_1,:) Q_1$ to a multiple of the first unit vector, that is, the trailing $r-1$ entries of $V(j_1,:)$ are annihilated. Then~\eqref{eq:orthprojection} is equivalent to updating $V \gets V Q_1$ and setting the first column of $V$ to zero. It will be convenient to carry out the second part only implicitly, by considering the last $r-1$ columns of the updated $V$ in the subsequent steps. Repeatedly applying the described procedure 
leads to Algorithm~\ref{alg:ARP}, which is illustrated in Figure~\ref{fig:picture}. {Let us emphasize that Algorithm~\ref{alg:ARP} is mathematically equivalent to Algorithm~\ref{alg:ARPnew} and it requires $\mathcal O(nr^2)$ operations.}
\begin{algorithm}
\textbf{Input: } Matrix $V \in \R^{n \times r}$ with orthonormal columns defining a row space approximation~\eqref{eq:rowspaceapprox} 

\textbf{Output: }Indices $J=(j_1, \ldots, j_r)$  defining a column subset selection 

    \begin{algorithmic}[1]
    \STATE{Initialize $J = ()$ and ${W}_0 = V$}
    \FOR{$k = 1,\ldots,r$}
        \STATE{Set $p_j = \|{W}_{k-1}(j,k:r)\|_2^2/(r-k+1)$ for $j = 1,\ldots,n$}\label{line:prob}
        \STATE{Sample index $j_k$ according to probabilities $p_j$}
        \STATE{Update $J \leftarrow (J, j_k)$}
        \STATE{Update $W_k \leftarrow {W}_{k-1} Q_k$ with Householder reflector $Q_k$ that annihilates ${W}_{k-1}(j_k,k+1:r)$ 
        }\label{line:householder}
    \ENDFOR
    \end{algorithmic}
    \caption{Adaptive Randomized Pivoting for CSSP (ARP)}
    \label{alg:ARP}
\end{algorithm}

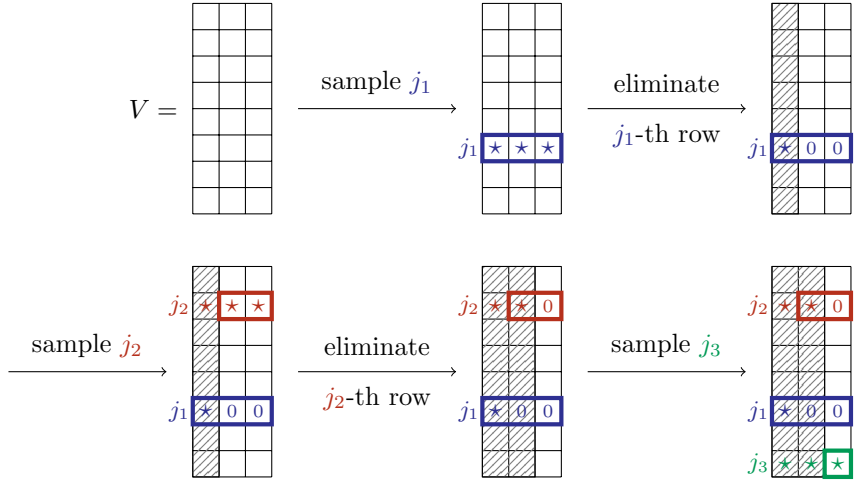
\begin{figure}
    \begin{center}
        \begin{tikzpicture}[scale=.35]
            \node at (-1.5,0) {$V = $};
            \draw (0,-4) grid (3,4);
            
            \draw[->] (4,0) -- (10,0);
            \node at (7,1) {sample \textcolor{Blue}{$j_1$}};
            \draw (11,-4) grid (14,4);
            \node at (10.5,-1.5) {\small \textcolor{Blue}{$j_1$}};
            \draw[ultra thick, Blue] (11,-2) rectangle (14,-1);
            \node[Blue] at (11.5,-1.5) {$\star$};
            \node[Blue] at (12.5,-1.5) {$\star$};
            \node[Blue] at (13.5,-1.5) {$\star$};
            
            \draw[->] (15,0) -- (21,0);
            \node at (18,1) {eliminate}; \node at (18,-1) {\textcolor{Blue}{$j_1$}-th row};
            \draw (22,-4) grid (25,4);
            \node at (21.5,-1.5) {\small \textcolor{Blue}{$j_1$}};
            \draw[pattern=north east lines, pattern color = gray] (22,-4) rectangle (23,4);
            \draw[ultra thick, Blue] (22,-2) rectangle (25,-1);
            \node[Blue] at (22.5,-1.5) {$\star$};
            \node[Blue] at (23.5,-1.5) {\scriptsize $0$};
            \node[Blue] at (24.5,-1.5) {\scriptsize $0$};
            
            \draw[->] (-7,-10) -- (-1,-10);
            \node at (-4,-9) {sample \textcolor{BrickRed}{$j_2$}};
            \draw (0,-14) grid (3, -6);
            \node at (-0.5,-11.5) {\small \textcolor{Blue}{$j_1$}};
            \node at (-0.5,-7.5) {\small \textcolor{BrickRed}{$j_2$}};
            \draw[pattern=north east lines, pattern color = gray] (0,-14) rectangle (1,-6);
            \draw[ultra thick, Blue] (0,-12) rectangle (3,-11);
            \node[Blue] at (0.5,-11.5) {$\star$};
            \node[Blue] at (1.5,-11.5) {\scriptsize $0$};
            \node[Blue] at (2.5,-11.5) {\scriptsize $0$};
            \draw[ultra thick, BrickRed] (1,-8) rectangle (3,-7);
            \node[BrickRed] at (0.5,-7.5) {$\star$};
            \node[BrickRed] at (1.5,-7.5) {$\star$};
            \node[BrickRed] at (2.5,-7.5) {$\star$};
            
            \draw[->] (4,-10) -- (10,-10);
            \node at (7,-9) {eliminate};
            \node at (7,-11) {\textcolor{BrickRed}{$j_2$}-th row};
            \draw (11,-14) grid (14,-6);
            \draw[pattern=north east lines, pattern color = gray] (11,-14) rectangle (13,-6);
            \node at (10.5,-11.5) {\small \textcolor{Blue}{$j_1$}};
            \node at (10.5,-7.5) {\small \textcolor{BrickRed}{$j_2$}};
            \draw[ultra thick, Blue] (11,-12) rectangle (14,-11);
            \node[Blue] at (11.5,-11.5) {$\star$};
            \node[Blue] at (12.5,-11.5) {\scriptsize $0$};
            \node[Blue] at (13.5,-11.5) {\scriptsize $0$};
            \draw[ultra thick, BrickRed] (12,-8) rectangle (14,-7);
            \node[BrickRed] at (11.5,-7.5) {$\star$};
            \node[BrickRed] at (12.5,-7.5) {$\star$};
            \node[BrickRed] at (13.5,-7.5) {\scriptsize $0$};
            
            \draw[->] (15,-10) -- (21,-10);
            \node at (18,-9) {sample \textcolor{ForestGreen}{$j_3$}};
            \draw (22,-14) grid (25,-6);
            \draw[pattern=north east lines, pattern color = gray] (22,-14) rectangle (24,-6);
            \node at (21.5,-11.5) {\small \textcolor{Blue}{$j_1$}};
            \node at (21.5,-7.5) {\small \textcolor{BrickRed}{$j_2$}};
            \node at (21.5,-13.5) {\small \textcolor{ForestGreen}{$j_3$}};
            \draw[ultra thick, Blue] (22,-12) rectangle (25,-11);
            \node[Blue] at (22.5,-11.5) {$\star$};
            \node[Blue] at (23.5,-11.5) {\scriptsize $0$};
            \node[Blue] at (24.5,-11.5) {\scriptsize $0$};
            \draw[ultra thick, BrickRed] (23,-8) rectangle (25,-7);
            \node[BrickRed] at (22.5,-7.5) {$\star$};
            \node[BrickRed] at (23.5,-7.5) {$\star$};
            \node[BrickRed] at (24.5,-7.5) {\scriptsize $0$};
            \draw[ultra thick, ForestGreen] (24,-14) rectangle (25,-13);
            \node[ForestGreen] at (22.5,-13.5) {$\star$};
            \node[ForestGreen] at (23.5,-13.5) {$\star$};
            \node[ForestGreen] at (24.5,-13.5) {$\star$};
        \end{tikzpicture}
    \end{center}
    \caption{Illustration of Algorithm~\ref{alg:ARP} applied to a matrix $V \in \R^{8 \times 3}$.}\label{fig:picture}
\end{figure}

The Householder reflector in Line~\ref{line:householder} of Algorithm~\ref{alg:ARP}
is an orthogonal matrix of the form
\begin{equation} \label{eq:formQk}
 Q_k = \begin{bmatrix}
        \Id_{k-1} & 0 \\
        0 & \Id_{n-k+1} - 2u_k u_k^T
       \end{bmatrix}
\end{equation}
for a suitably chosen unit vector $u_k$ such that
\begin{equation}\label{eq:propertyVk}
    {W}_k(j_k,k:r)  = {W}_{k-1}(j_k,k:r)(\Id-2u_k u_k^T) = {W}_k(j_k,k)e_1^T = \|{W}_{k-1}(j_k,k:r)\|_2 e_1^T.
\end{equation}
In particular, the first $k-1$ columns of ${W}_{k-1}$ are not modified when multiplying with $Q_k$, and previously introduced zeros are preserved. The latter implies that the probabilities of previously chosen indices are zero and, hence, no index is chosen twice.
As ${W}_{k-1}$ is obtained from $V$ by orthonogonal transformations, it remains an orthonormal basis. In particular, we have that $\|{W}_{k-1}(:,k:r)\|_F^2 = r-k+1$, which confirms that the quantities $p_j \ge 0$ computed in Line~\ref{line:prob} indeed represent a probability distribution. The final, updated matrix ${W}_r$ satisfies
\[
 {W}_r = V Q, \quad Q:= Q_1 Q_2 \cdots Q_r.
\]
In particular, we have that
\begin{equation} \label{eq:VVr}
 {W}_r(J,:) = V(J,:) Q.
\end{equation}
By construction, ${W}_r(J,:)$ is lower triangular. Its $k$th diagonal element contains the norm of the row ${W}_{k-1}(j_k,k:r)$ reduced in the $k$th step. By the choice of $p_j$, this norm is always nonzero. This allows us to conclude that ${W}_r(J,:)$ and, in turn, $V(J,:)$ are always invertible. {Finally, note that the rank-$r$ approximation~\eqref{eq:lowrankapprox} can be written as $A(:,J)W_r(J,:)^{-T}W_r^T$, and the fact that $W_r(J,:)$ is lower triangular allows one to conveniently compute the factor $W_r(J,:)^{-T} W_r^T$ by backward substitution.}

%% file: analysis.tex
\subsection{Analysis}\label{sec:analysis}

To establish the error bound~\eqref{eq:mainresultintro} for Algorithm~\ref{alg:ARPnew} {(and therefore Algorithm~\ref{alg:ARP})}, we first analyze the error of the approximation~\eqref{eq:lowrankapprox}.  {Defining the oblique projector
\begin{equation} \label{eq:obliquedef2}
 \widetilde \Pi_J:= \Id - E_J ( V^T E_J)^{-1} V^T
\end{equation}
allows us to express the approximation error as
$
 A - A(:,J) V(J,:)^{-T} V^T = A \widetilde \Pi_J.
$
From~\eqref{eq:obliquedef2},
it immediately follows that $A \widetilde \Pi_J E_J = 0$. In other words, the columns $j_1,\ldots, j_r$ of $A$ are ``interpolated''.} The following lemma establishes {an important property} of $\widetilde \Pi_J$ { as well as a decomposition into a product of $r$ simpler oblique projections, each corresponding to a step of the algorithm}.
\begin{lemma} \label{lemma:remarkableprops2}
 Consider $\widetilde \Pi_J$ as in~\eqref{eq:obliquedef2} for an index set $J = (j_1,\ldots, j_r)$ such that $V(J,:) = E_J^T V$ is invertible. Then
 \begin{equation} \label{eq:remarkableprop1new}
  (\Id - VV^T) \widetilde \Pi_J = \widetilde \Pi_J.
 \end{equation}
 {Moreover, if we define}
 \begin{equation}\label{eq:widetildepik}
{\widetilde \Pi_k := \Id - e_{j_k} \frac{V_{k-1}(j_k,:) V_{k-1}^T}{\|V_{k-1}(j_k,:)\|_2^2 } \text{ for } k = 1, \ldots, r,}
 \end{equation}
 {which are oblique projections, we have}
 \begin{equation} \label{eq:remarkableprop2new}
    { \widetilde \Pi_1 \widetilde \Pi_2 \cdots \widetilde \Pi_k = \Id - E_{J_k} \left ( V^T E_{J_k} \right )^\dagger V^T,}
 \end{equation}
 {in particular $\widetilde \Pi_J = \widetilde \Pi_1 \cdots \widetilde \Pi_r$.}
\end{lemma}
\begin{proof}
 The first property~\eqref{eq:remarkableprop1new} is verified by straightforward calculation:
 \begin{equation*}
     (\Id-VV^T) \widetilde \Pi_J = \Id - VV^T - E_J(V^TE_J)^{-1}V^T + VV^TE_J(V^TE_J)^{-1}V^T = \widetilde \Pi_J.
 \end{equation*}
 {We have
\begin{equation*}
    \widetilde \Pi_k^2 = \Id - 2 e_{j_k} \frac{V_{k-1}(j_k,:) V_{k-1}^T}{\|V_{k-1}(j_k,:)\|_2^2} + e_{j_k} \frac{V_{k-1}(j_k,:) V_{k-1}(j_k,:)^T V_{k-1}(j_k,:)V_{k-1}^T}{\|V_{k-1}(j_k,:)\|_2^4} = \widetilde \Pi_k,
\end{equation*}
which shows that the matrices $\widetilde \Pi_k$ are oblique projections.}
To show the property~\eqref{eq:remarkableprop2new}, we first note that {$V_{k-1}(j_{\ell},:) = 0$ for $\ell < k$, because it has been explicitly zeroed out by the algorithm, which} implies $\widetilde \Pi_k e_{j_\ell} = e_{j_\ell}$. For $\ell = k$, we have
$\widetilde \Pi_k e_{j_k} = 0$. Taken together, these relations establish the interpolation property
\begin{equation} \label{eq:interpPJnew}
 \widetilde \Pi_1 \cdots \widetilde \Pi_{{k}} E_{{J_k}} = 0.
\end{equation}
On the other hand, the {construction of the matrices $V_1, \ldots, V_k$ and the }form of the factors $\widetilde \Pi_k$ {imply} that
$
 \widetilde \Pi_1 \cdots \widetilde \Pi_{{k}} = \Id - E_{{J_k}} X V^T
$
for some ${k \times k}$ matrix $X$. Multiplying this expression with $E_{{J_k}}$ and using~\eqref{eq:interpPJnew} implies {$X = ( V^T E_{{J_k}})^\dagger$}, which shows~\eqref{eq:remarkableprop2new}. 
\end{proof}

Setting $\widetilde A := A (\Id-VV^T)$, the first property~\eqref{eq:remarkableprop1new} of Lemma~\ref{lemma:remarkableprops2} establishes
\begin{equation} \label{eq:AandAtilde}
 A - A(:,J)V(J,:)^{-T}V^T = A \widetilde \Pi_J =  \widetilde A \widetilde \Pi_J = \widetilde A - \widetilde A(:,J)V(J,:)^{-T}V^T.
\end{equation}
This relation is quite remarkable, because it provides a one-to-one correspondence between column selection for the original matrix $A$ and column selection for the residual matrix $\widetilde A$. 
It is also the basis for both, Osinsky's analysis~\cite[Theorem 1]{Osinsky2023} and our main result.
\begin{theorem}\label{thm:mainthm}
    Let $A \in \R^{m \times n}$ and let $V \in \R^{n \times r}$ be an orthonormal basis. Then the random index set $J$ returned by Algorithm~\ref{alg:ARPnew} satisfies
    \begin{equation}\label{eq:mainresult}
        \mathbb{E}\left [\|A - A(:,J) V(J,:)^{-T} V^T\|_F^2\right ] = (r+1) \|A - AVV^T\|_F^2.
    \end{equation}
\end{theorem}

\begin{proof}
In view of~\eqref{eq:AandAtilde}, establishing~\eqref{eq:mainresult}  is equivalent to showing
$\mathbb{E}[\|\widetilde A \widetilde \Pi_J\|_F^2] = (r+1) \|\widetilde A\|_F^2$.
We now define $\widetilde A_{0} := \widetilde A$ and, recursively,
\begin{equation}  \label{eq:defwaknew}
 \widetilde A_{k} := \widetilde A_{k-1} \widetilde\Pi_k = \widetilde A_{k-1} - \widetilde A_{k-1}(:,j_k) { \frac{V_{k-1}(j_k,:) V_{k-1}^T}{\|V_{k-1}(j_k,:)\|_2^2}}.
\end{equation}
By~\eqref{eq:remarkableprop2new}, $\widetilde A_{r} = \widetilde A \widetilde \Pi_J$.
{We claim that the rows of $\widetilde A_k$ are orthogonal to the columns of $V_k$. Indeed, we have $$\widetilde A_k = \widetilde A_0 \widetilde \Pi_1 \widetilde \Pi_2 \cdots \widetilde \Pi_k = A(\Id - VV^T)\left (\Id - E_{J_k}(V^T E_{J_k})^\dagger V^T \right )$$
where the first equality follows from the definition of $\widetilde A_k$ and the second follows from~\eqref{eq:remarkableprop2new} and the definition of $\widetilde A_0$. Then we have
\begin{align*}
    \widetilde A_k V_k &= A(\Id - VV^T)\left (\Id - E_{J_k}(V^T E_{J_k})^\dagger V^T \right ) V\left ( \Id - V^T E_{J_k} (V^TE_{J_k})^\dagger \right ) \\
    & = A(\Id - VV^T)(V - E_{J_k}(V^TE_{J_k})^\dagger V^TV) \left ( \Id - V^T E_{J_k} (V^TE_{J_k})^\dagger \right )\\
    & = A(V-VV^TV) - A(\Id-VV^T)E_{J_k}(V^T E_{J_k})^\dagger \left ( \Id - V^T E_{J_k} (V^TE_{J_k})^\dagger \right ) = 0,
\end{align*}
where the first equality follows from Lemma~\ref{lem:propV} and the rest follows from orthonormality of the columns of $V$ and properties of pseudoinverses.}

Using conditional expectation, it follows that
\begin{align}
& \mathbb{E}[\|\widetilde A_k\|_F^2 \mid j_1, \ldots, j_{k-1}] = \mathbb{E}[\|\widetilde A_{k-1} \widetilde\Pi_k\|_F^2 \mid j_1, \ldots, j_{k-1}] \nonumber\\
& =  \mathbb{E} \left [\|\widetilde A_{k-1}\|_F^2 + \left\|\widetilde A_{k-1}(:,j_k) {\frac{V_{k-1}(j_k,:) V_{k-1}^T}{\|V_{k-1}(j_k,:)\|_2^2}}   \right\|_F^2 \mid j_1, \ldots, j_{k-1} \right ] \nonumber \\
& { =  \mathbb{E} \left [\|\widetilde A_{k-1}\|_F^2 +  \frac{\|\widetilde A_{k-1}(:,j_k)\|_2^2}{\|V_{k-1}(j_k,:)\|_2^2} \mid j_1, \ldots, j_{k-1} \right ]} \label{eq:normAktildenew}\\
& =  \|\widetilde A_{k-1}\|_F^2 + \frac{1}{r-k+1} \sum_{j \notin J_{k-1}} \|\widetilde A_{k-1}(:,j)\|_2^2 = \frac{r-k+2}{r-k+1}  \| \widetilde A_{k-1} \|_F^2.\nonumber
\end{align}
Here, the second equality follows from~\eqref{eq:defwaknew} using Pythagoras, because the rows of $\widetilde A_{k-1}$ are orthogonal to the {columns of $V_{k-1}$; to obtain the third equality we used Lemma~\ref{lem:propV} to compute}
\begin{align*}
    {\|V_{k-1}(j_k,:)V_{k-1}^T\|_2} & {= \|V(j_k,:) \left(\Id - V^T E_{J_{k-1}}(V^T E_{J_{k-1}})^\dagger\right )^2 V^T\|_2 }\\
    & {= \|V(j_k,:) \left (\Id - V^T E_{J_{k-1}}(V^T E_{J_{k-1}})^\dagger\right )\|_2 = \|V_{k-1}(j_k,:)\|_2};
\end{align*}
{the fourth equality follows from the definition of the sampling probabilities in Line~\ref{line:prob} of {Algorithm~\ref{alg:ARPnew}}}. Using the law of total expectation, we obtain that
\begin{align*}
    \mathbb{E}[\| \widetilde A_r \|_F^2] & = \mathbb{E}[\mathbb{E}[\| \widetilde A_{r-1} \|_F^2 \mid j_1,\ldots,j_{k-1}]] = \frac{2}{1} \mathbb{E}[\|\widetilde A_{r-1}\|_F^2] \\
    &= \cdots = \frac{3}{2}\cdot \frac{2}{1} \mathbb{E}[\|\widetilde A_{r-2}\|_F^2]   = \cdots =  (r+1) \| \widetilde A_0\|_F^2,
\end{align*}
which completes the proof using $A \widetilde \Pi_J =  \widetilde A \widetilde \Pi_J = \widetilde A_r $.
\end{proof}

\color{black}
By Jensen's inequality, Theorem~\ref{thm:mainthm} implies that
\begin{align*}
    \mathbb{E}[\|A - A(:,J)V(J,:)^{-T} V^T\|_F] & \le \left ( \mathbb{E}[\|A - A(:,J)V(J,:)^{-T} V^T\|_F^{{2}}]\right )^{1/2} \\
    & =\sqrt{r+1} \|A - AVV^T\|_F.
\end{align*}
Moreover, Markov's inequality allows us to turn the second moment bound of Theorem~\ref{thm:mainthm} into a tail bound. For example, with probability at least $99\%$ we have that 
\begin{equation*}
    \|A-A(:,J)V(J,:)^{-T}V^T\|_F \le 10\sqrt{r+1}\|A-AVV^T\|_F.
\end{equation*}
Because orthogonal projection onto the column space of $A(:,J)$ produces the minimal error, one obtains the same bounds
for $\mathbb{E}[\|A - \Pi_J A\|_F^2]$. In particular, when choosing {$V = \Vopt$, which contains right singular vectors belonging to $r$ largest singular values} of $A$, we obtain a corollary that matches~\eqref{eq:bestapproximation} in expectation.
\begin{corollary} \label{cor:best}
    The random index set $J$ returned by Algorithm~\ref{alg:ARP} applied to $V = \Vopt$ satisfies
    \begin{align*}
        \mathbb{E}[\|A - \Pi_J A\|_F^2]  \le \mathbb{E}[\|A - A(:,J) \Vopt(J,:)^{-T} \Vopt^T\|_F^2] 
         \le (r+1)\big( \sigma_{r+1}^2(A) + \ldots + \sigma_n^2(A) \big),
    \end{align*}
    {where $\Pi_J$ denotes the orthogonal projector onto the column space of $A(:, J)$.}
\end{corollary}

Let us emphasize that there is no need to use the optimal $V$ for getting good upper bounds. {In fact, one can use} the randomized SVD~\cite{rsvd} with a bit of oversampling to compute $V$ {and get a bound that is only moderately larger compared to the one by Corollary~\ref{cor:best}.}
{\begin{corollary}
Consider the following procedure:
    \begin{enumerate}
 \item Sample Gaussian random matrix $\Omega \in \R^{m \times (r+2)}$ and compute $Y = A^T \Omega$.
 \item Compute orthonormal basis $V \in \R^{n \times (r+2)}$ by performing QR decomposition of $Y$.
 \item Obtain index set $J$ of cardinality $r+2$ by applying Algorithm~\ref{alg:ARP} to $V$.
\end{enumerate}
Then the resulting index set $J$ satisfies
\begin{equation*}
    \mathbb{E}[\|A - \Pi_J A\|_F^2] \le (r+3)(r+1) \left ( \sigma_{r+1}^2 + \cdots + \sigma_n^2(A) \right ).
\end{equation*}
\end{corollary}
\begin{proof}
We have
\begin{align*}
    \mathbb{E}_{J,\Omega}[\|A - \Pi_J A\|_F^2] & = \mathbb{E}_{\Omega}\left [ \mathbb{E}_J [A - \Pi_J A \|_F^2 \mid \Omega] \right ] = (r+3) \mathbb{E}_\Omega[\|A - AVV^T\|_F^2] \\
    & \le (r+1)(r+3)\left ( \sigma_{r+1}^2 + \cdots + \sigma_n^2(A) \right ),
\end{align*}
where the first equality follows from the law of total expectation, the second equality follows from Theorem~\ref{thm:mainthm}, and the final inequality follows from Theorem 10.5 in~\cite{rsvd} and its proof.
\end{proof}}

The slightly increased cardinality of the index set could be avoided by performing a recompression of $A VV^T$. However, this would require multiplying $A$ with $V$. The plain procedure above comes with the benefit that only $A^T \Omega$ needs to be computed. By using a structured random matrix (e.g., a subsampled randomized trigonometric transform), this step can be accelerated, potentially significantly. 

\subsection{Derandomization of Algorithm~\ref{alg:ARP}}\label{sec:derandomizedCSSP}

By \emph{derandomizing} the proof of Theorem~\ref{thm:mainthm}, it is possible to recover Osinsky's algorithm~\cite[Algorithm 1]{Osinsky2023} and develop a deterministic version of Algorithm~\ref{alg:ARP} that outputs an index set $J$ such that
\begin{equation}\label{eq:detCSSP}
\|A - A(:,J)V(J,:)^{-T} V^T\|_F^2 \le (r+1) \|A - AVV^T\|_F^2.
\end{equation}
Assume that $j_1, \ldots, j_{k-1}$ have already been (deterministically) selected, and the corresponding matrix $\widetilde A_{k-1} = \widetilde A_0 \widetilde \Pi_{1} \cdots \widetilde \Pi_{k-1}$ has been computed explicitly, with the oblique projections defined as in~\eqref{eq:widetildepik}. 
At the $k$th step, choosing
\begin{equation}\label{eq:jk}
    j_k \in \arg\min_{j} \|\widetilde A_{k-1}(:,j)\|_2^2 / \|V_{k-1}(j,{:})\|_2^2
\end{equation}
minimizes $\|\widetilde A_k\|_F^2$ given $\widetilde A_{k-1}$; see~\eqref{eq:normAktildenew}. Because the minimum cannot be larger than the expected value, the inequality~\eqref{eq:normAktildenew} implies that 
\begin{equation}\label{eq:det}
\|\widetilde A_k\|_F^2 \le \frac{r-k+2}{r-k+1}\|\widetilde A_{k-1}\|_F^2.
\end{equation}
Iterating the inequalities~\eqref{eq:det} for $k=1, \ldots, r$ proves~\eqref{eq:detCSSP}. 

{To write~\eqref{eq:jk} in terms of the matrices $W_k$ appearing in Algorithm~\ref{alg:ARP} it is sufficient to note that $\|V_{k-1}(j,:)\|_2 = \|W_{k-1}(j, k:r)\|_2$ for all indices $j$ and for $k=1,\ldots,r$. Moreover, the oblique projections $\widetilde \Pi_k$ can be written in terms of the matrices $W_k$ as
\begin{equation}\label{eq:PiktildeW}
    \widetilde \Pi_k = \Id - e_{j_k} \frac{W_k(:,k)^T}{W_k(j_k,k)} \quad \text{ for } k = 1, \ldots, r.
\end{equation}}
The described derandomization of Algorithm~\ref{alg:ARP} is summarized in Algorithm~\ref{alg:Osinsky}. 

\begin{algorithm}
\textbf{Input: } Matrix $A \in \R^{m \times n}$ and matrix $V \in \R^{n \times r}$ with orthonormal columns defining a row space approximation~\eqref{eq:rowspaceapprox}

\textbf{Output: }Indices $J=(j_1, \ldots, j_r)$  defining a column subset selection of $A$

    \begin{algorithmic}[1]
    \STATE{Compute $\widetilde A_0 \leftarrow A - AVV^T$}\label{line:Atilde}
    \STATE{{Initialize $J_0 \leftarrow ()$ and $W_0 \leftarrow V$}}
    \FOR{$k = 1,\ldots,r$}
        \STATE{Set $j_k \in \arg\min_{j} \|\widetilde A_{k-1}(:,j)\|_2^2 / \|{W}_{k-1}(j,k:r)\|_2^2$}
        \STATE{{Set $J_k \leftarrow (J_{k-1}, j_k)$}}
        \STATE{Update ${W}_k \leftarrow {W}_{k-1} Q_k$ with Householder reflector $Q_k$ that annihilates ${W}_{k-1}(j_k,k+1:r)$
        \STATE{Update $\widetilde A_k \leftarrow \widetilde A_{k-1} \widetilde \Pi_k$ {using~\eqref{eq:PiktildeW}}}\label{line:updateAtilde}
        }
    \ENDFOR
    \STATE{{Set $J \leftarrow J_r$}}
    \end{algorithmic}
    \caption{Osinsky's deterministic algorithm for CSSP~\cite[Algorithm 1]{Osinsky2023}}
    \label{alg:Osinsky}
\end{algorithm}

\begin{remark}[Computational cost of randomized vs. derandomized algorithms]\label{rmk:costosinsky}
The deterministic guarantee~\eqref{eq:detCSSP} comes at the expense of a higher computational cost compared to Algorithm~\ref{alg:ARP}. In particular, the need for computing and updating the column norms of the residual matrix $\widetilde A_{k-1}$ is relatively costly. Already for the initial residual $\widetilde A_{0}$, this requires computing the column norms of $A$ and multiplying $A$ with $V$. As the column norms are subjective to subtractive cancellation in the subsequent computations, they need to be determined quite accurately, which essentially means that
the \emph{whole} matrix $A$ needs to be explicitly available. In contrast, Algorithm~\ref{alg:ARP} only requires access to the columns in $J$, once $V$ is available. When combined with a cheap procedure to determine $V$ (e.g., by the randomized SVD with a highly structured random matrix $\Omega$), Algorithm~\ref{alg:ARP} becomes an appealing alternative to Algorithm~\ref{alg:Osinsky}.
\end{remark}

In view of the disadvantages mentioned in Remark~\ref{rmk:costosinsky}, it is tempting to derandomize Algorithm~\ref{alg:ARP} differently, in a way that does not require access to $A$: At each step, we select $j_k = \arg\max_j p_j$, with $p_j$ defined by the updated row norms in Line~\ref{line:prob} of Algorithm~\ref{alg:ARP}. However, the following example shows that such a  greedy approach does not necessarily yield good results.
\begin{example}[Greedy on $V$ is not enough!]\rm   Consider the matrix
    \begin{equation*}
        A = \begin{bmatrix} 1 & 0 \\ 0 & 10^{-4} \end{bmatrix} \begin{bmatrix} \frac{2}{\sqrt{n+3}} & - \frac{1}{\sqrt{n+3}} & - \frac{1}{\sqrt{n+3}} & \cdots & -\frac{1}{\sqrt{n+3}} \\ \sqrt{\frac{n-1}{n+3}} & \frac{2}{\sqrt{(n-1)(n+3)}}& \frac{2}{\sqrt{(n-1)(n+3)}} & \cdots & \frac{2}{\sqrt{(n-1)(n+3)}}\end{bmatrix},
    \end{equation*}
    and $r = 1$. By~\eqref{eq:bestapproximation}, there is a column index $j$ such that
    \begin{equation}\label{eq:wanted}
    \|A-A(:,j)A(:,j)^\dagger A\|_F^2 \le 2 \cdot 10^{-8}.
    \end{equation}
    Using $V = \begin{bmatrix} \frac{2}{\sqrt{n+3}} & - \frac{1}{\sqrt{n+3}} & - \frac{1}{\sqrt{n+3}} & \cdots & -\frac{1}{\sqrt{n+3}} \end{bmatrix}^T$ as a rank-$1$ row space approximation of $A$, the greedy choice explained above would choose the first index $j=1$ as this is the element of largest magnitude in $V$. When taking, e.g., $n = 10,000$, one has that
    \begin{equation*}
        \|A - A(:,j)A(:,j)^\dagger A\|_F^2 \approx \begin{cases}
            2.5 \cdot 10^{-5} & \text{ if } j = 1,\\
            10^{-8} & \text{ if } j \neq 1.
        \end{cases}
    \end{equation*}
    Choosing $j_1 = 1$ therefore gives a column subset selection error that is significantly worse than what is predicted by~\eqref{eq:wanted}. In contrast, Algorithm~\ref{alg:ARP} does \emph{not} select $j_1 = 1$ with probability $(n+1)/(n+3)$, which is very close to $1$ for large values of $n$.
\end{example}

%% file: deim.tex
\section{DEIM}\label{sec:DEIM}

Algorithm~\ref{alg:ARP} is an attractive choice for selecting indices in the Discrete Empirical Interpolation Method (DEIM), which is a popular method for approximating nonlinear functions in reduced-order modelling~\cite{Chaturantabut2010}. Given a function $f\equiv f(\xi) \in \R^n$ that is expensive to evaluate and known to be well approximated by an orthonormal basis $V \in \R^{n \times r}$, in the sense of $f \approx VV^Tf$, the goal of DEIM is to construct a good approximation of $f$ that does not require the evaluation of the whole function. For this purpose, 
DEIM selects a set $I$ of $r$ indices and returns the approximation
\[
 f \approx V (E_I^T V)^{-1} E_I^T f = V (E_I^T V)^{-1} f(I),
\]
that is, only the part of $f$ contained in $I$ needs to be evaluated.
 {Letting $\| \cdot \|_2$ denotes the spectral norm of a matrix,} the quantity $\|(E_I^T V)^{-1}\|_2 = \|V(I,:)^{-1}\|_2$ gives an indication of the quality of the approximation because
\begin{equation*}
    \|f - V (E_I^T V)^{-1} E_I^T f\|_2  \le \|(E_I^TV)^{-1}\|_2 \|f - VV^T f\|_2;
\end{equation*}
see~\cite[Lemma 3.2]{Chaturantabut2010}. It is known that there \emph{exists} an index set $I$ such that $\|(E_I^TV)^{-1}\|_2 \le \sqrt{1+r(n-r)}$; see~\cite[Theorem 2.1]{Drmac2016}.

We suggest to apply Algorithm~\ref{alg:ARP} to $V$ in order to determine the index set $I$ for DEIM. As a corollary of Theorem~\ref{thm:mainthm}, this approach enjoys guarantees -- in expectation -- on the random index set $I$ that match or exceed the best deterministic bound mentioned above. 
\begin{corollary} \label{cor:deim}
    Given an orthonormal basis $V \in \R^{n \times r}$, the index set returned by Algorithm~\ref{alg:ARP} satisfies
    \begin{equation}\label{eq:deimbound}
    \mathbb{E}[\|f - V (E_I^T V)^{-1} E_I^T f\|_2^2] \le (r+1)\| f - VV^Tf\|_2^2
    \end{equation}
    for an arbitrary but fixed vector $f \in \R^n$.
    Moreover, \[
\mathbb{E}[\|(E_I^TV)^{-1}\|_F^2] = r(n-r+1), \quad 
\mathbb{E}[\|(E_I^TV)^{-1}\|_2^2]  \le r(n-r)+1.
\]
\end{corollary}
\begin{proof}
The result~\eqref{eq:deimbound} follows from applying Theorem~\ref{thm:mainthm} to $A := f^T$. Theorem~\ref{thm:mainthm} applied to $A = \Id_n$ gives
\begin{equation}\label{eq:id}
    \mathbb{E}[\|\widetilde \Pi_I\|_F^2] = (r+1)\|\Id-VV^T\|_F^2 = (r+1)(n-r).
\end{equation}
Moreover, we have
\begin{align}
    \| \widetilde \Pi_I \|_F^2 & = \|\Id - E_I(V^TE_I)^{-1}V^T\|_F^2 = \|\Id - E_I E_I^T\|_F^2 + \|E_I(E_I^T-(V^TE_I)^{-1}V^T)\|_F^2\nonumber\\
    & = n-r + \|E_I^T - (V^T E_I)^{-1}V^T\|_F^2 = n-2r+\|(V^TE_I)^{-1}\|_F^2, \label{eq:boundpiJ}
\end{align}
where the last equality follows from the fact that $(V^T E_I)^{-1}V^T$ has an identity block in the columns corresponding to $I$ and subtracting $E_I^T$ annihilates these columns; therefore $\|(V^T E_I)^{-1} V^T\|_F^2 = r + \|E_I^T - (V^TE_I)^{-1}V^T\|_F^2$. Combining~\eqref{eq:id} with~\eqref{eq:boundpiJ} we get $\mathbb{E}[\| (E_I^T V)^{-1}\|_F^2] = r(n-r+1)$.
Finally, we have
\begin{align*}
    \|(E_I^TV)^{-1}\|_2^2& = \|(E_I^TV)^{-1}\|_F^2 - \sum_{k=2}^r \sigma_{r}((E_I^TV)^{-1})^{2} \le \|(E_I^TV)^{-1}\|_F^2-(r-1)
\end{align*}
since all the singular values of $E_I^TV$ are upper bounded by the singular values of $V$, which are equal to $1$. Taking expectations, $\mathbb{E}[\|(E_I^TV)^{-1}\|_2^2] \le \mathbb{E}[ \|(E_I^TV)^{-1}\|_F^2]-r+1 = 1+r(n-r)$.
\end{proof}

Let us stress that Algorithm~\ref{alg:ARP} is oblivious to $f$, that is, it does not need to access the usually unknown vector $f$ in order to satisfy the favorable bound~\eqref{eq:deimbound}. In contrast, Osinsky's algorithm (Algorithm~\ref{alg:Osinsky}) requires access to $f$ in order to achieve the bound~\eqref{eq:deimbound} deterministically. As suggested in~\cite{Osinsky2023},
Osinsky's algorithm can be applied to the identity matrix instead of $f$ for determining $I$, but this increases the pre-factor from $r+1$ to $r(n-r)+1$.

%% file: crossapproximation.tex
\section{CUR and cross approximation} \label{sec:crossapproximation}

As highlighted in the introduction, CSSP is equivalent to finding a rank-$r$ approximation of $A$ such that the left rank-$r$ factor consists of $r$ selected columns of $A$.
The general goal of CUR and related concepts is to construct rank-$r$ approximations for which also the right factor is constructed from $r$ selected rows of $A$. Often, a CUR
approximation is understood in the more specific form
\begin{equation}\label{eq:cur}
    A \approx A(:,J) A(:,J)^\dagger A A(I,:)^\dagger A(I,:),
\end{equation}
where $I$ and $J$ are (well chosen) row and column index sets of cardinality $r$.
A simple trick, many times used in the literature (see~\cite{DongMartinsson2023,Mahoney2009,Sorensen2016} for examples), can be used to relate CSSP and DEIM to CUR. Applied in our setting, this
trick amounts to using Algorithm~\ref{alg:ARP} with a matrix $V$ defining a row space approximation~\eqref{eq:rowspaceapprox} to select $J$ and again Algorithm~\ref{alg:ARP} with a matrix $U$ defining a column space approximation of $A$ (i.e. a row space approximation~\eqref{eq:rowspaceapprox} of $A^T$) to select $I$. Pythagoras' theorem gives
\begin{equation*}
    \|A - A(:,J)A(:,J)^\dagger A A(I,:)^\dagger A(I,:)\|_F^2 = \|A-\Pi_J A\|_F^2 + \|\Pi_J (A - A A(I,:)^\dagger A(I,:))\|_F^2,
\end{equation*}
{where we recall that $\Pi_J$ denotes the orthogonal projector onto the column space of $A(:, J)$. Now, }
Theorem~\ref{thm:mainthm} implies that
\begin{equation*}
    \mathbb{E}[\|A - A(:,J)A(:,J)^\dagger A A(I,:)^\dagger A(I,:)\|_F^2] \le (r+1)\left (\|A(\Id-VV^T)\|_F^2 + \|(\Id-UU^T)A\|_F^2\right ).
\end{equation*}
In contrast, the (non-adaptive) subspace sampling method from~\cite{Mahoney2009} uses the same input but needs to sample $\mathcal O(\epsilon^{-2} r \log r)$ columns and rows in order to achieve a relative Frobenius norm error $2+\epsilon$ whp.

A major disadvantage of CUR~\eqref{eq:cur} is that one needs to sample the full matrix, or at least needs to be able to compute matrix-vector products with $A$, in order to build the middle matrix $A(:,J)^\dagger A A(I,:)^\dagger$. The cross approximation, also called skeleton approximation, avoids this by choosing a different middle matrix:
\begin{equation}
 A \approx A(:,J)A(I,J)^{-1}A(I,:). \label{eq:crossapproximation2}
\end{equation}
This approximation exists whenever one selects index sets $I$ and $J$ of cardinality $r$ such that $A(I,J)$ is invertible. Also, note that $A$ is matched exactly (interpolated) in the selected rows and columns. To build the approximation~\eqref{eq:crossapproximation2}, only these rows and columns of $A$ need to be evaluated. {In particular, unlike for CUR, no additional matrix-vector products are needed.} Combined with a cheap procedure for sampling $I$ and $J$, this gives, in principle, the possibility to obtain an approximation with a complexity $\mathcal O(r (m+n))$, which is sublinear in the size of $A$. For example, under a rather restrictive incoherence assumption on the singular vectors, uniform sampling can be safely used~\cite{ChiuDemanet2013}.

In the spirit of~\cite[Theorem 2]{Osinsky2023} and~\cite[Algorithm 1]{DongMartinsson2023}, we sample the index sets $I$ and $J$ defining~\eqref{eq:crossapproximation2} as follows.
Once a matrix $V$ corresponding to a row space approximation~\eqref{eq:rowspaceapprox} is known and column indices $J$ have been computed with Algorithm~\ref{alg:ARP}, we obtain an orthonormal basis $Q_J$ of $A(:,J)$ as a row space approximation~\eqref{eq:rowspaceapprox} of $A^T$. We apply Algorithm~\ref{alg:ARP} once more, now to $Q_J$, to obtain an index for selecting columns of $A^T$, that is, rows of $A$. The resulting procedure is summarized in Algorithm~\ref{alg:CA}. 

\begin{algorithm}
\textbf{Input: } Matrix $A\in \R^{m\times n}$ and an orthonormal basis $V \in \R^{n \times r}$ defining a row space approximation~\eqref{eq:rowspaceapprox}

\textbf{Output: }Index sets $I = (i_1, \ldots, i_r)$ and $J=(j_1, \ldots, j_r)$  defining a cross approximation~\eqref{eq:crossapproximation2}

    \begin{algorithmic}[1] 
    \STATE{Obtain index set $J$ by applying Algorithm~\ref{alg:ARP} to $V$}
    \STATE{Compute an orthonormal basis $Q_J$ of $A(:,J)$ by a QR decomposition}
    \STATE{Obtain index set  $I$ by applying Algorithm~\ref{alg:ARP} to $Q_J$}
    \end{algorithmic}
    \caption{Adaptive Randomized Pivoting for cross approximation (ARPcross)}
    \label{alg:CA}
\end{algorithm}

Using Theorem~\ref{thm:mainthm} we can establish a bound on the quality of the cross approximation returned by Algorithm~\ref{alg:CA}.
\begin{theorem}\label{thm:CA}
    Let $A \in \R^{m \times n}$ and let $V \in \R^{n \times r}$ be an orthonormal basis. Then the random index sets $I$ and $J$ returned by Algorithm~\ref{alg:CA} satisfy
    \begin{equation*}
        \mathbb{E}\left [\|A - A(:,J) A(I,J)^{-1}A(I,:)\|_F^2\right ] \le (r+1)^2 \|A - AVV^T\|_F^2.
    \end{equation*}
\end{theorem}

\begin{proof}
From the first property of Lemma~\ref{lemma:remarkableprops2}, it follows that
\begin{align*}
    A - A(:,J)A(I,J)^{-1} A(I,:) & = (\Id - A(:,J) (E_I^T A(:,J))^{-1} E_I^T ) A
    = (\Id - Q_J (E_I^T Q_J)^{-1} E_I^T)A,
\end{align*}
where $Q_J$ is an orthonormal basis of $A(:,J)$. 
Using the law of total expectation, we get
\begin{align*}
    \mathbb{E}[\|A - A(:,J)A(I,J)^{-1}A(I,:)\|_F^2] &=
      \mathbb{E}_J\left [\mathbb{E}_I[ \|(\Id - Q_J (E_I^T Q_J)^{-1} E_I^T) A\|_F^2 \mid J ]\right ] \\
     &= \mathbb{E}_J [(r+1) \|(\Id-Q_J Q_J^T)A\|_F^2] \\
     &\le (r+1)\mathbb{E}_J[\|A - A(:,J)V(J,:)^{-T}V^T\|_F^2] \\ &= (r+1)^2 \|A- AVV^T\|_F^2,
\end{align*}
As in the proof of Corollary~\ref{cor:deim}, the 
second equality follows from applying Theorem~\ref{thm:mainthm} to $A^T$ (with its row space approximation $Q_J$), the inequality follows from the minimality of orthogonal projections, and the last equality  follows from again applying Theorem~\ref{thm:mainthm}, now to the matrix $A$.
\end{proof}

Algorithm~\ref{alg:CA} can be derandomized in a similar way to Algorithm~\ref{alg:ARP} and this recovers~\cite[Theorem 2]{Osinsky2023}. However, as for Algorithm~\ref{alg:Osinsky}, this derandomized algorithm comes with the disadvantage that the full matrix $A$ needs to {be} accessed repeatedly. In contrast, Algorithm~\ref{alg:CA} only needs to access $r$ rows and columns of $A$. It enjoys sublinear complexity $\mathcal O(r^{{2}} (m+n))$, \emph{once} a row space approximation $V$ is
available.

%% file: cholesky.tex
\section{Nystr\"om approximation for SPSD matrices} \label{sec:spsd}
For an SPSD matrix $A\in \R^{n \times n}$, it is sensible to choose $I = J$ in the cross approximation~\eqref{eq:crossapproximation2}, which gives 
rise to the Nystr\"om approximation~\cite{Williams2000}:
\begin{equation}\label{eq:symmetricCA}
    A \approx A(:,J) A(J,J)^{-1} A(J,:).
\end{equation}
The \emph{Gram correspondence}~\cite{EpperlyBlog} establishes a one-to-one relation between~\eqref{eq:symmetricCA} and CSSP. To see this, let $J$ be an arbitrary index set of cardinality $r$ such that $A(J,J)$ is invertible. Then for any matrix $B$ such that $B^TB=A$ we have
\begin{equation*}
    A(:,J)A(J,J)^{-1}A(J,:) = B^T B(:,J)(B(:,J)^TB(:,J))^{-1}B(:,J)^T B.
\end{equation*}
Because $\Pi := B(:,J)\left (B(:,J)^TB(:,J)\right )^{-1}B(:,J)^T$ is an orthogonal projection (onto the column space of $B(:,J)$), it holds that 
$(\Id-\Pi) = (\Id-\Pi)^2$ and, hence,
\begin{equation}\label{eq:Gram}
        \|A - A(:,J)A(J,J)^{-1}A(J,:)\|_* = \|B^T(\Id-\Pi)^2B\|_*  = \|(\Id-\Pi)B\|_F^2.
    \end{equation}
    Here, $\| \cdot \|_*$ denotes the nuclear norm of a matrix (i.e., the sum of its singular values), which coincides with the trace for an SPSD matrix. 
    The relation~\eqref{eq:Gram} allows us to turn results on CSSP for $B$ into results on Nystr\"om~\eqref{eq:symmetricCA} for $A$. For example, the result~\eqref{eq:bestapproximation} implies the existence of an index set $J$ such that 
\begin{equation}\label{eq:SPSD}
    \|A - A(:,J) A(J,J)^{-1} A(J,:)\|_* \le (r+1)(\sigma_{r+1}(A) + \ldots + \sigma_n(A));
\end{equation}
see also~\cite[Theorem 1]{Massei2022}. {Let us emphasize that the factor $(r+1)$ in~\eqref{eq:SPSD} is tight. Indeed, Proposition 3.3 in~\cite{Deshpande2006} states the existence, for any $\epsilon > 0$, of an $(r+1)\times (r+1)$ matrix $B$ such that, for any index set $J$ of cardinality $r$, $\|B - \Pi_J B\|_F^2 \ge (1-\epsilon)(r+1)\sigma_{r+1}^2(B)$. This allows us to construct, for any $\epsilon > 0$, a corresponding matrix $A := B^T B$, for which the relation~\eqref{eq:Gram} implies
\begin{equation*}
\|A - A(:,J)A(J,J)^{-1}A(J,:)\|_* \ge (1-\epsilon)(r+1)\sigma_{r+1}^2(B) =  (1-\epsilon)(r+1)\sigma_{r+1}(A).
\end{equation*}}

\subsection{Randomized algorithm} 

Our randomized Nystr\"om approximation simply applies 
Algorithm~\ref{alg:ARP} to a row space approximation $V$ of $A$
to get an index set $J$ defining~\eqref{eq:symmetricCA}.
The relation~\eqref{eq:Gram} allows us to leverage Theorem~\ref{thm:mainthm} for obtaining an error bound. 

\begin{corollary}\label{cor:SPSD}
    Let $A \in \R^{n \times n}$ be SPSD and let $V \in \R^{n \times r}$ be an orthonormal basis. Then the random index set $J$ returned by Algorithm~\ref{alg:ARP} satisfies
    \begin{equation*}
        \mathbb{E}[\|A - A(:,J)A(J,J)^{-1}A(J,:)\|_*] \le (r+1) \|(\Id-VV^T)A(\Id-VV^T)\|_*.
    \end{equation*}
\end{corollary}
\begin{proof}
As above, let $B \in \R^{n \times n}$ be such that $B^TB = A$. By taking expectation on both sides of~\eqref{eq:Gram}, we obtain
    \begin{align*}
        & \mathbb{E}[\|A - A(:,J)A(J,J)^{-1}A(J,:)\|_*]  = \mathbb{E}[\|B - B(:,J)(B(:,J)^TB(:,J))^{-1}B(:,J)^T B\|_*]\\
        & \le \mathbb{E}[\| B - B(:,J)V(J,:)^{-T}V^T\|_F^2]= (r+1)\|B - BVV^T\|_F^2\\
        & = (r+1)\|(\Id-VV^T)B^TB(\Id-VV^T)\|_* = (r+1)\|(\Id-VV^T)A(\Id-VV^T)\|_*,
    \end{align*}
    where we used Theorem~\ref{thm:mainthm} for the second equality.
\end{proof}

For $V = \Vopt$, we have $\|(\Id-VV^T)A(\Id-VV^T)\|_* = \sigma_{r+1}(A) + \ldots + \sigma_n(A)$ and, in this case, the bound of Corollary~\ref{cor:SPSD} matches~\eqref{eq:SPSD} in expectation.

The recently proposed randomly pivoted Cholesky algorithm~\cite{Chen2023} determines the index set by adaptive sampling according to the diagonal of the updated matrix $A$. It comes with the major advantage that only the diagonal and the selected rows/columns of the (original) matrix $A$ need to be accessed. In particular, there is no need for providing a row space approximation $V$. On the other hand, Theorem 5.1 in~\cite{Chen2023} requires oversampling to guarantee the error bound~\eqref{eq:SPSD} (in expectation), especially if the best rank-$r$ approximation error is small. In practice, however, such oversampling does not seem to be needed.

\subsection{Derandomization} \label{sec:SPSDderandomized}

While the definition~\eqref{eq:crossapproximation2} of cross approximation coincides with Nystr\"om for $I = J$, neither the 
randomized Algorithm~\ref{alg:CA} for cross approximation nor its derandomized version from~\cite{Osinsky2023} are guaranteed to preserve this relation for SPSD $A$. Once again, the Gram correspondence is useful in order to derive a suitable derandomized version of the algorithm from the previous section.

The basic recipe for the derandomized algorithm is simple. For any matrix $B$ such that $B^TB = A$, apply 
Algorithm~\ref{alg:Osinsky} to determine an index set $J$ such that
\begin{equation*}
    \|B - B(:,J)V(J,:)^{-T}V^T\|_F^2 \le (r+1) \|B - B VV^T\|_F^2,
\end{equation*}
By the proof of Corollary~\ref{cor:SPSD}, it follows that
this index set $J$ also satisfies
\begin{equation*} 
\|A - A(:,J)A(J,J)^{-1}A(J,:)\|_* \le (r+1) \|(\Id-VV^T)A(\Id-VV^T)\|_*.
\end{equation*}
This recipe is not yet practical because it requires the potentially expensive computation of a square root or Cholesky factor $B$ of $A$. 

Remarkably, it is possible to devise an algorithm for $A = B^T B$ that is mathematically equivalent to applying Algorithm~\ref{alg:Osinsky} to $B$, but does not require the computation of $B$.
To see this, let $\widetilde B_0 = B(\Id - VV^T)$ and $\widetilde B_k := \widetilde B_{k-1} \widetilde \Pi_k$, in analogy to~\eqref{eq:defwaknew}, and set $\widetilde A_k = \widetilde B_k^T \widetilde B_k$ for $k = 1, \ldots, r$. To choose the index according to~\eqref{eq:jk}, we need to compute the squared column norms of $\widetilde B_k$ in each step. Because 
\begin{equation*}
    \|\widetilde B_k(:,j)\|_2^2 = \widetilde A_k(j,j),
\end{equation*}
it suffices to compute the diagonal of $\widetilde A_k$. For this purpose, we first compute the diagonal entries of the initial matrix $\widetilde A_0 = (\Id-VV^T)A(\Id-VV^T)$ by performing the product $Y := AV$ plus $\mathcal O(nr^2)$ additional operations. In each subsequent step, the 
diagonal entries are updated via the relation 
\begin{equation}\label{eq:updatediagelements}
\widetilde A_k(j,j) = e_j^T \widetilde \Pi_k^T \widetilde A_{k-1} \widetilde \Pi_k e_j  =  \widetilde A_{k-1}(j, j) - \frac{2\widetilde A_{k-1}(j,j_k){W}_k(j,k)}{{W}_k(j_k,k)} + \frac{\widetilde A_{k-1}(j_k,j_k) {W}_k(j,k)^2}{{W}_k(j_k,k)^2},
    \end{equation}
    for $j=1,\ldots,n$, {where $\widetilde \Pi_k$ is as in~\eqref{eq:PiktildeW}. } 
    To perform the update~\eqref{eq:updatediagelements} we need access to the $j_k$th column of $\widetilde A_{k-1}$, which can be computed from the relation
\begin{align}\label{eq:newsampledcolumn}
        \widetilde A_{k-1} e_{j_k} & = \widetilde \Pi_{k-1}^T \widetilde \Pi_{k-2}^T \cdots \widetilde \Pi_1^T (\Id - VV^T)A(\Id-VV^T) \widetilde \Pi_1 \cdots \widetilde \Pi_{k-2} \widetilde \Pi_{k-1} e_{j_k}\nonumber\\
        & {= (\Id - V(E_{J_{k-1}}^TV)^\dagger E_{J_{k-1}}^T)(\Id - VV^T)A(\Id - VV^T)(\Id - E_{J_{k-1}}(V^T E_{J_{k-1}})^\dagger V^T)e_{j_k}.}
    \end{align}
    {The relation~\eqref{eq:newsampledcolumn} allows us to compute the $j_k$th column of $\widetilde A_{k-1}$ in $\mathcal O(nr)$ operations: one can keep and update a QR factorization of $V^TE_{J_{k-1}}$ so that applying the oblique projection $E_{J_{k-1}}(V^T E_{J_{k-1}})^\dagger V^T$ -- or its transpose -- to any vector costs at most $\mathcal O(nr)$ operations, and the product $Y = AV$ has already been computed at the beginning.} 
Therefore, also~\eqref{eq:newsampledcolumn} can be evaluated efficiently and without the need of working explicitly with $B$.

The described procedure is summarized in Algorithm~\ref{alg:derandomizedSPSD}. The diagonal entries of $\widetilde A_k$ are denoted by $d_1, \ldots, d_n$ and updated at each step; the vectors $w_k$ are equal to $\widetilde A(:,j_k)$.  The computational cost is $\mathcal O(nr^2)$ plus the cost of forming the matrix $Y = AV$. Therefore, Algorithm~\ref{alg:derandomizedSPSD} is more expensive than Algorithm~\ref{alg:ARP} because it needs the computation of $AV$, but less expensive than Algorithm~\ref{alg:Osinsky} (see Remark~\ref{rmk:costosinsky}) because the computation of the column norms of $\widetilde B_k$ reduce{s} to the computation of the diagonal elements of $\widetilde A_k$, which can be done efficiently. As Algorithm~\ref{alg:derandomizedSPSD} is mathematically equivalent to Algorithm~\ref{alg:Osinsky} applied to $B$, we have the following result.
\begin{algorithm}
\textbf{Input: } SPSD matrix $A \in \R^{n \times n}$ and orthonormal basis $V \in \R^{n \times r}$ defining a row space approximation

\textbf{Output: }Indices $J=(j_1, \ldots, j_r)$  defining Nystr\"om approximation of $A$

    \begin{algorithmic}[1]
    \STATE{Compute $Y \leftarrow AV$}
     \STATE{Compute diagonal elements $d_1, \ldots, d_n$ of $\widetilde A_0 = A-VY^T-YV^T+V(V^TY)V^T$}\label{line:initd}
     \STATE{Initialize {$J_0 \leftarrow ()$} and $V_0 \leftarrow V$}
    \FOR{$k = 1,\ldots,r$}
        \STATE{Choose $j_k \in \arg\min_{j} d_j / \|{W}_{k-1}(j,k:r)\|_2^2$}
        \STATE{{Set $J_k \leftarrow (J_{k-1}, j_k)$}}
        \STATE{Update ${W}_k \leftarrow {W}_{k-1} Q_k$ with Householder reflector $Q_k$ that annihilates ${W}_{k-1}(j_k,k+1:r)$}\label{line:Qk}
        \STATE{Compute $w_k \leftarrow \widetilde A_{k-1}e_{j_k}$ via~\eqref{eq:newsampledcolumn}}%
        \STATE{Update $d_j \leftarrow d_j - 2\frac{w_k(j){W}_k(j,k)}{{W}_k(j_k,k)} + \frac{w_k(j_k) {W}_k(j,k)^2}{{W}_k(j_k,k)^2}$ for $j=1,\ldots,n$}
    \ENDFOR
    \STATE{{Set $J \leftarrow J_r$}}
    \end{algorithmic}
    \caption{Nystr\"om with deterministic adaptive pivoting}
    \label{alg:derandomizedSPSD}
\end{algorithm}
\begin{corollary} \label{cor:detguaranteeSPSD}
    Let $A \in \R^{n \times n}$ be SPSD and let $V \in \R^{n \times r}$ be an orthonormal basis. Then the index set $J$ returned by Algorithm~\ref{alg:derandomizedSPSD} satisfies
    \[
        \|A - A(:,J)A(J,J)^{-1}A(J,:)\|_* \le (r+1) \|(\Id-VV^T)A(\Id-VV^T)\|_*.
    \]
\end{corollary}
For $\Vopt = V$, we obtain $\|A - A(:,J)A(J,J)^{-1}A(J,:)\|_* \le (r+1) (\sigma_{r+1}(A) + \ldots + \sigma_n(A))${, which is the best possible guarantee} for a deterministic Nystr\"om approximation. For example, the error bound for the greedy approach, which amounts to Cholesky with diagonal pivoting, features a pre-factor that grows exponentially with $r$~\cite{Harbrecht2012}.

%% file: experiments.tex
\section{Numerical experiments} \label{sec:experiments}

In this section, we illustrate the behavior of our newly proposed algorithms for CSSP, DEIM, cross approximation, and Nystr\"om approximation on several test matrices. As mentioned in the introduction, the main purpose of the numerical experiments is to validate our methods. The fine-tuning, the development of efficient implementations, and a broad application study of our proposed algorithms are beyond the scope of this work. All algorithms were implemented and executed in MATLAB R2024b on a laptop with an Intel Core Ultra 5 125U $\times$ 14 CPU and 16 GB of RAM. To simplify the setting, we choose the input matrix $V$ for our algorithms to be $\Vopt$, the matrix containing the first right singular vectors, computed with a singular value decomposition of the data. Each randomized algorithm is run for $100$ times and we plot the average error, with the error bars indicating the interval that discards the $10\%$ best and $10\%$ worst errors. {The code to reproduce the numerical experiments is available at \url{https://github.com/Alice94/ARP}.}

\subsection{CSSP for a sparse matrix}\label{sec:CSS_large}

As a first example, we consider a sparse matrix $A$ of size $4,282 \times 8,617$ arising from a linear programming problem sequence, as considered in~\cite{DongMartinsson2023} and available from the SuiteSparse matrix collection~\cite{Davis2011}\footnote{Downloaded from \url{http://sparse.tamu.edu/Meszaros/large}.}. We compare Algorithm~\ref{alg:ARP} (ARP) with leverage score sampling, column norm sampling, column pivoted QR (CPQR), and Osinsky's deterministic algorithm for CSSP (Algorithm~\ref{alg:Osinsky}). {Insisting on selecting \emph{exactly} $r$ columns is not the usual setting for leverage score and column norm sampling, which use oversampling or only attain $r$ columns in expectation. In order to make a meaningful comparison, we make the following adjustment: we start with sampling $r$ columns according to the given probability distribution, then we remove the potential duplicates, update the probabilities and sample again, until we reach the desired cardinality of exactly $r$ distinct columns. This strategy allows us to fairly compare ARP to these sampling strategies. } The results are shown in Figure~\ref{fig:large}. Interestingly, leverage score sampling (which also utilizes the matrix $\Vopt$)  performs relatively badly; all other strategies are very close to the best low-rank approximation in this example.

\subsection{CSSP for genetic data}\label{sec:CSS_DNA}

The second example we consider for CSSP uses the GSE10072 cancer genetics data set from the National Institutes of Health, considered in~\cite[Example 3]{Sorensen2016}\footnote{Downloaded from~\url{https://ftp.ncbi.nlm.nih.gov/geo/series/GSE10nnn/GSE10072/matrix/}.}. The matrix $A$ has size $22,283 \times 107$.  Again, we compare ARP with leverage score sampling, column norm sampling, CPQR, and Algorithm~\ref{alg:Osinsky}. The results are reported in Figure~\ref{fig:DNA}. All methods perform comparably well on this example; Algorithm~\ref{alg:Osinsky} performs slightly better than the other ones, at the expense of a higher computational cost.

\begin{figure}[htb]
\begin{subfigure}[b]{0.48\textwidth}
\centering
    \includegraphics[scale=.53]{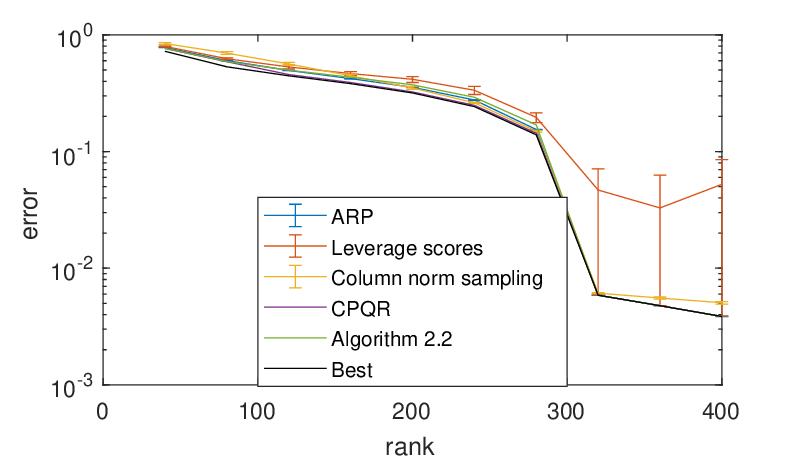}
    \caption{Matrix from Section~\ref{sec:CSS_large}.}
    \label{fig:large}
    \end{subfigure}
    \begin{subfigure}[b]{0.48\textwidth}
\centering
    \includegraphics[scale=.53]{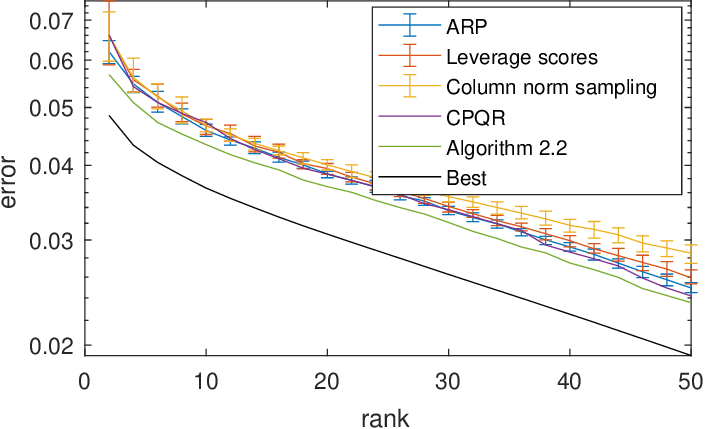}
    \caption{Matrix from Section~\ref{sec:CSS_DNA}.}
    \label{fig:DNA}
\end{subfigure}
\caption{Comparison of different CSSP strategies; the error is $\|A - A(:,J)A(:,J)^\dagger A\|_F / \|A\|_F$.}
\label{fig:example_CSSP}
\end{figure}

\subsection{DEIM}\label{sec:example_deim}
Given values of $x_1,x_2,\mu_1,\mu_2 \in[0,1]$, we define the following function, which was considered in~\cite[Section 2.3]{Peherstorfer2014} and~\cite[Section 6.1]{Saibaba2020}:
\begin{align*}
    f(x_1,x_2,\mu_1,\mu_2) & = g(x_1,x_2,\mu_1,\mu_2) + g(1-x_1,1-x_2,1-\mu_1,1-\mu_2) \\
    & + g(1-x_1,x_2,1-\mu_1,\mu_2) + g(x_1, 1-x_2, \mu_1, 1-\mu_2), 
\end{align*}
where
\begin{equation*}
g(x_1,x_2,\mu_1,\mu_2) = \left ( ((1-x_1)-(0.99\mu_1-1))^2 + ((1-x_2)-(0.99\mu_2-1))^2 + 0.1^2  \right )^{-1/2}.
\end{equation*}
We discretize the function on a $50 \times 50$ grid in $[0,1]^2$ for the spatial variables $x_1,x_2$ and on a $12 \times 12$ grid for the parameters $\mu_1,\mu_2$. We arrange all the function values into 
a $2500 \times 144$ matrix $A$ such that its column and row indices refer to the discretized 
spatial variables and parameters, respectively. To test the quality of index sets $I$ of cardinality $r$, analogously to~\cite{Peherstorfer2014}, we construct vectors $f_1, \ldots, f_{121} \in \R^{2500}$ obtained by discretizing the function on a $50 \times 50$ grid in $[0,1]^2$ for parameters $\mu_1,\mu_2$ which are taken on a $11 \times 11$ equispaced grid in $[0,1]^2$. We then consider the average relative $L^2$-error of the DEIM approximation of these new samples:
\begin{equation}\label{eq:errorDEIM}
\frac{1}{121} \sum_{j=1}^{121} \frac{\|f_j - V V(I,:)^{-1} f_j(I)\|_2}{\|f_j\|_2}.
\end{equation}
We compare the proposed strategy -- ARP -- with uniform sampling, Q-DEIM~\cite{Drmac2016}, and R-DEIM~\cite[Algorithm 6]{Saibaba2020}. R-DEIM first samples $2r$ random columns using leverage score sampling and then subselects $r$ rows using strong rank-revealing QR~\cite{Gu1996}. The results are shown in Figure~\ref{fig:DEIM} (left). While Q-DEIM, which is a greedy deterministic strategy, 
gives the best results in this example, it can give very suboptimal results in certain cases, for instance when applied to a modified Kahan matrix~\cite[Example 1]{Gu1996}. ARP achieves, on average, a slightly larger error compared to R-DEIM, with the advantage that ARP is simpler and does not require oversampling.

\subsection{Cross approximation}\label{sec:example_CA}
Let us consider the nonsymmetric kernel matrix $A \in \R^{2000 \times 2000}$ obtained from
discretizing the kernel function
\begin{equation*}
K(\alpha,\beta) = \exp\left ( -15\sqrt{\alpha^2+\beta^2} \right ) + \exp\left (-75\sqrt{(\alpha-1)^2 + (\beta-1)^2} \right )
\end{equation*}
by choosing $2000$ equispaced samples for $\alpha \in [0,1]$ and $2000$ uniformly random samples for $\beta \in [0,1]$. 

In Figure~\ref{fig:CA}, we compare the column and row selection strategy of Algorithm~\ref{alg:CA} (ARPcross) with a standard implementation of Adaptive Cross Approximation (ACA) with full pivoting (``ACA full'') and ACA with partial pivoting (ACA partial)~\cite{Bebendorf2000}. {We also compare with the Randomized LUPP algorithm from~\cite[Algorithm 2]{DongMartinsson2023} (randLUPP), where the size of the sketch is chosen to be equal to the target rank of the cross approximation.} Given index sets $I$ and $J$, the error is measured in the Frobenius norm as $\|A - A(:,J) A(I,J)^{-1} A(I,:)\|_F / \|A\|_F$. This matrix represents two Gaussian ``bumps'' and is a challenging example for ACA with partial pivoting: One starts with selecting a row and column corresponding to one of the two bumps, and it takes some time until the algorithm detects the other bump. ACA with full pivoting, {randLUPP,} and ARPcross perform comparably well on this example. {The main strength of ARPcross over randLUPP is its strong theoretical guarantee.}

\begin{figure}[htb]
\begin{subfigure}[b]{0.48\textwidth}
\centering
    \includegraphics[scale=.6]{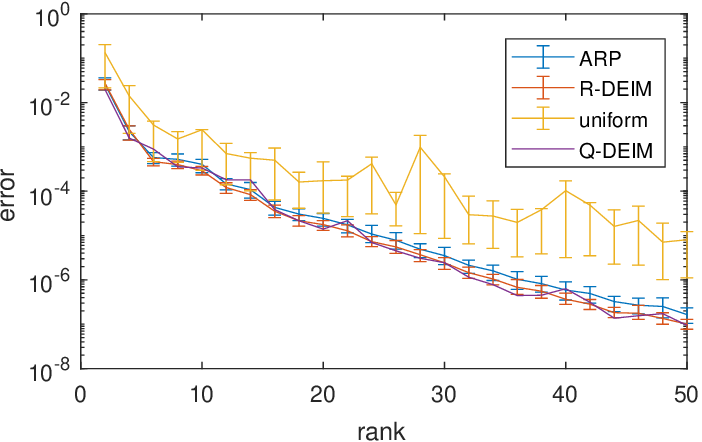}
    \caption{Comparison of strategies to select indices for DEIM for the function described in Section~\ref{sec:example_deim}. The error is measured as in~\eqref{eq:errorDEIM}.}
    \label{fig:DEIM}
    \end{subfigure}
    \begin{subfigure}[b]{0.48\textwidth}
\centering
    \includegraphics[scale=.6]{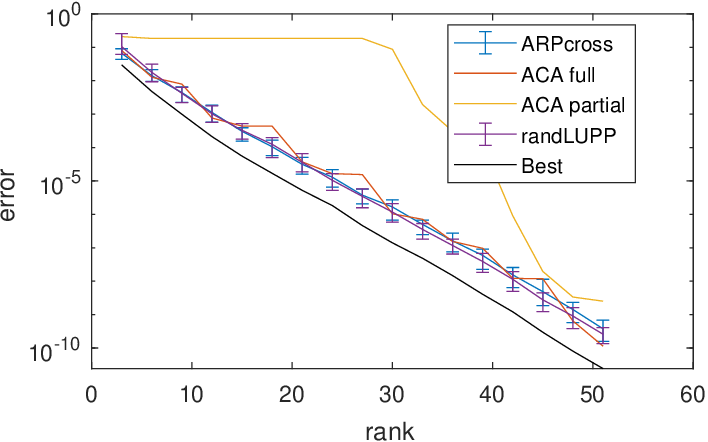}
    \caption{{Comparison of strategies for column and row selection for cross approximation on the matrix from Section~\ref{sec:example_CA}. }}
    \label{fig:CA}
\end{subfigure}
\caption{Numerical results for DEIM (left) and cross approximation (right).}
\label{fig:example_DEIM_CA}
\end{figure}

\subsection{Nystr\"om approximation for SPSD matrices}\label{sec:example_SPSD}
We consider two Gaussian kernel matrices taken from~\cite[Section 2.4]{Chen2023}\footnote{Code adapted from~\url{https://github.com/eepperly/Randomly-Pivoted-Cholesky}}. The first matrix represents a ``smile'' and is  constructed from $1000$ data points depicting a smile in $\R^2$. The kernel bandwidth is set to $2$ and the smile is located in $[-10 , 10] \times [-10 , 10]$. The second matrix represents a ``spiral'' and it is constructed from $1000$ data points in $\R^2$ depicting the logarithmic spiral $(e^{0.2t} \cos t, e^{0.2t} \sin t)$ for non-equispaced parameter values $t\in[0 , 64]$. The kernel bandwidth is $5$.

In Figure~\ref{fig:example_SPSD} we compare the approximation error $\|A - A(:,J)A(J,J)^{-1}A(J,:)\|_* / \|A\|_*$ corresponding to the index set $J$ returned by ARP with the approximation error corresponding to indices selected by RPCholesky~\cite{Chen2023}, leverage score sampling, and uniform sampling. We also compare with {three deterministic strategies: Algorithm~\ref{alg:derandomizedSPSD}, ACA with full pivoting (which in this case reduces to diagonal pivoting), and a greedy algorithm for nuclear norm maximization from~\cite[Algorithm 2]{Fornace2024}.} For the smile, ARP, ACA with full pivoting, and RPCholesky perform similarly, and they are better than both, leverage score and uniform sampling. For the spiral, the slow singular value decay is well-known to create difficulties for greedy methods, such as ACA with full pivoting; the error of ARP is visually below all the other randomized methods for this matrix. In both examples, the derandomized variant, Algorithm~\ref{alg:derandomizedSPSD}, {is usually comparable with greedy nuclear norm maximization, and} usually returns an even better index set compared to ARP, at the cost of a higher computational complexity. 

\begin{figure}[htb]
\begin{subfigure}[b]{0.37\textwidth}
\centering
    \includegraphics[scale=.28]{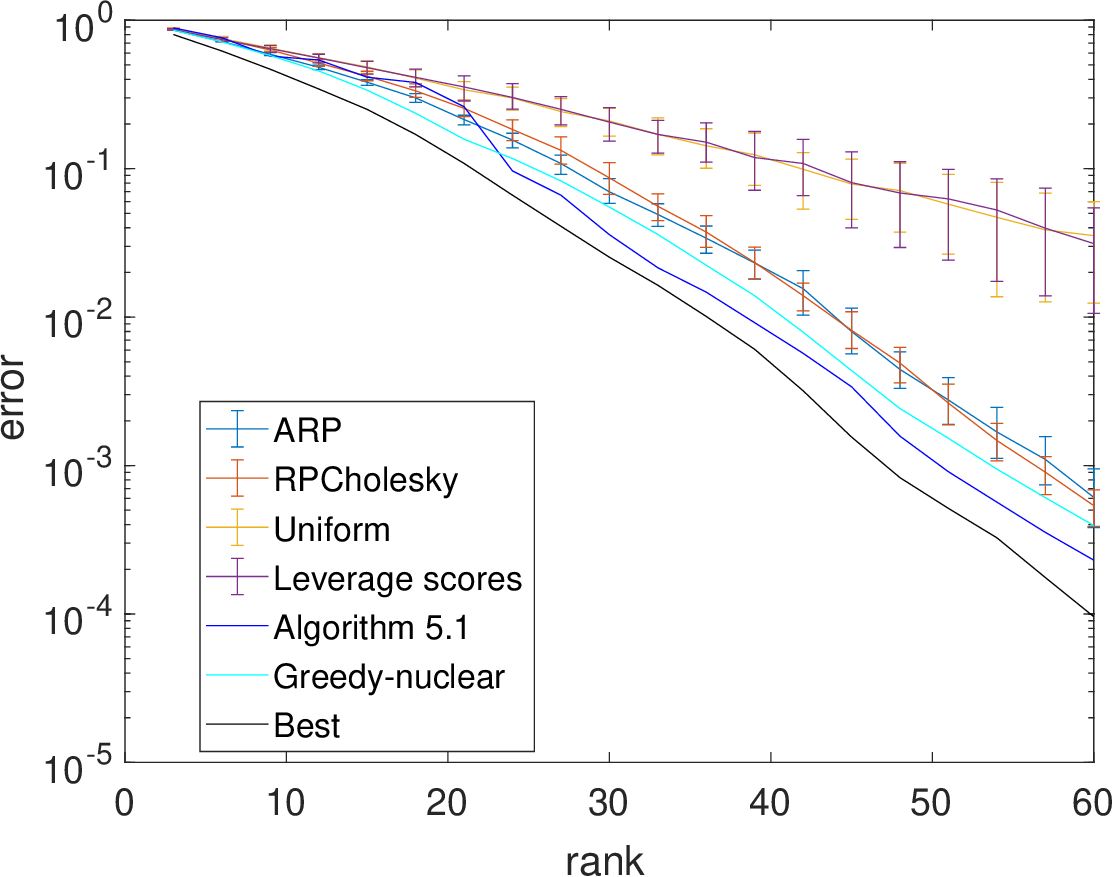}
    \end{subfigure}
    \begin{subfigure}[b]{0.22\textwidth}
    \centering
    \includegraphics[scale=.55]{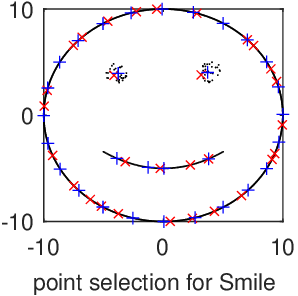}
    \end{subfigure}
    \begin{subfigure}[b]{0.37\textwidth}
\centering
    \includegraphics[scale=.28]{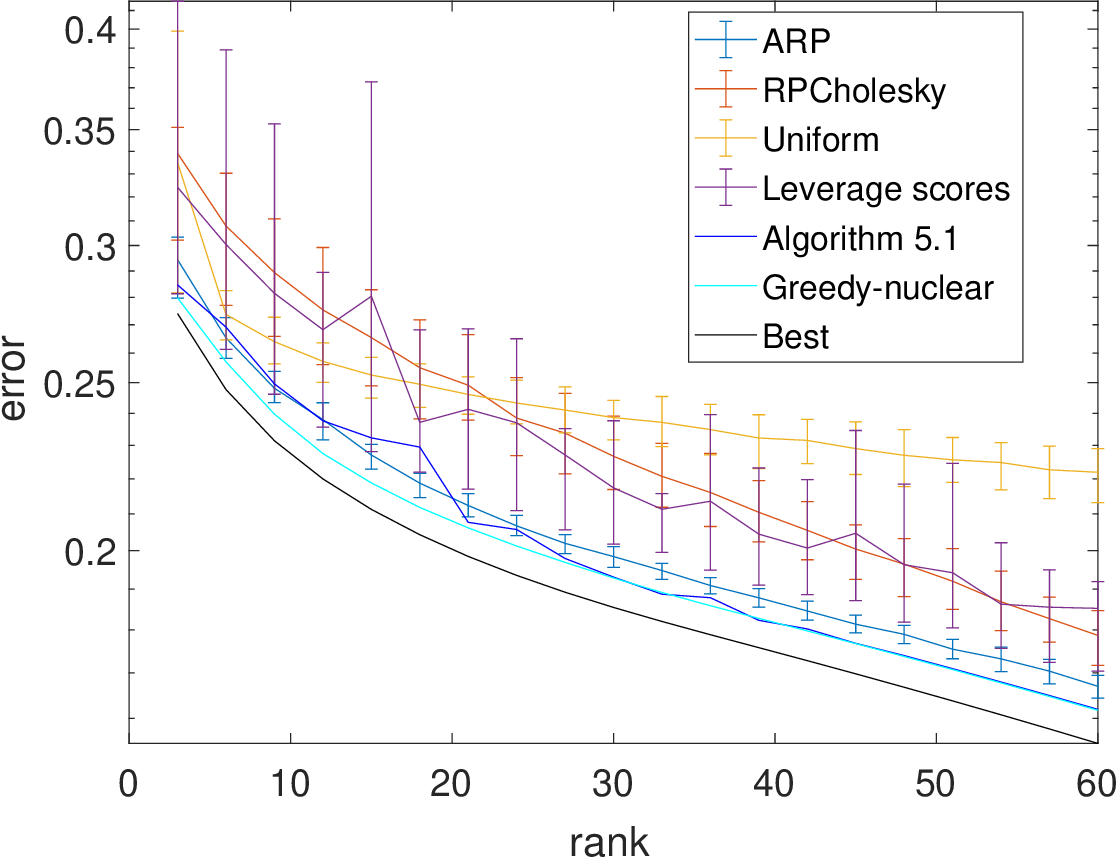}
\end{subfigure}
\caption{{Selection of column indices for SPSD matrices: smile (left) and spiral (right); see Section~\ref{sec:example_SPSD} for details. The figure in the middle illustrates the choice of points made by (one instance of) ARP (\textcolor{red}{red $\times$}) and Algorithm~\ref{alg:derandomizedSPSD} (\textcolor{blue}{blue $+$}) with rank $30$. }}
\label{fig:example_SPSD}
\end{figure}